\documentclass[graybox]{svmult}

\usepackage{mathptmx}       
\usepackage{helvet}         
\usepackage{courier}        
\usepackage{makeidx}         
\usepackage{graphicx}        
\usepackage{multicol}        
\usepackage[bottom]{footmisc}

\usepackage{amsmath}

\usepackage{amssymb}
\usepackage{marginnote}
\usepackage{todonotes}

\newcommand{\bx}{\overline{x}}
\newcommand{\by}{\overline{y}}
\newcommand{\bv}{\overline{v}}

\newcommand{\Cech}{\v{C}ech}

\newcommand{\R}{\mathbb{R}}

\newcommand{\ds}{\displaystyle}

\newtheorem{thm}{Theorem}[section]

\newtheorem{prop}{Proposition}[section]

\newcommand {\mm}[1] {\ifmmode{#1}\else{\mbox{\(#1\)}}\fi}

\newcommand{\Sspace}        {\mm{{\mathbb S}}}

\newcommand\Dg{\mathrm{Dg}}

\newcommand{\IC}{\mathsf{C}}

\newcommand{\rank}{\mm{\mathrm{rank}}}
\newcommand{\bdr}{\mm{\mathrm{\partial}}}

\definecolor{darkblue}{rgb}{0.0, 0.0, 0.8}
\definecolor{darkred}{rgb}{0.8, 0.0, 0.0}
\definecolor{darkgreen}{rgb}{0.0, 0.8, 0.0}

\newcommand{\ddelta}		{{d_{\IC_\delta}}}
\newcommand{\newGhat}	{{\tilde{G}}}

\graphicspath{{./figs/}}

\makeindex             

\begin{document}

\title*{A Complete Characterization of the 1-Dimensional Intrinsic \Cech{} Persistence Diagrams for Metric Graphs}
\titlerunning{1-D Intrinsic \Cech{} Persistence Diagrams for Metric Graphs} 
\author{Ellen Gasparovic, Maria Gommel, Emilie Purvine, Radmila Sazdanovic, \\Bei Wang, Yusu Wang, and Lori Ziegelmeier}
\authorrunning{Gasparovic, Gommel, Purvine, Sazdanovic, Wang, Wang, and Ziegelmeier} 
\institute{Ellen Gasparovic \at Union College, 807 Union Street, Schenectady, NY 12309, \email{gasparoe@union.edu}
\and Maria Gommel \at University of Iowa, 14 MacLean Hall, Iowa City, IA 52242,
\email{maria-gommel@uiowa.edu}
\and Emilie Purvine \at Pacific Northwest National Laboratory, 1100 Dexter Ave N., Seattle, WA 98109,\\ \email{emilie.purvine@pnnl.gov}
\and Radmila Sazdanovic \at NC State University, PO Box 8205, Raleigh NC 27695, \email{rsazdan@ncsu.edu}
\and Bei Wang \at University of Utah, 72 South Central Campus Drive,
Salt Lake City, UT 84112,\\  \email{beiwang@sci.utah.edu}
\and Yusu Wang \at The Ohio State University, 2015 Neil Ave., Columbus, OH 43210, \email{yusu@cse.ohio-state.edu}
\and Lori Ziegelmeier \at Macalester College, 1600 Grand Avenue, Saint Paul, MN 55104, \email{lziegel1@macalester.edu}}
%
%
\maketitle


\abstract{Metric graphs are special types of metric spaces used to model and represent simple, ubiquitous, geometric relations in data such as biological networks, social networks, and road networks. We are interested in giving a qualitative description of metric graphs using topological summaries. In particular, we provide a complete characterization of the $1$-dimensional intrinsic \Cech{} persistence diagrams for finite metric graphs using persistent homology. Together with complementary results by Adamaszek et al., which imply results on intrinsic \Cech{} persistence diagrams in all dimensions for a single cycle, our results constitute important steps toward characterizing intrinsic \Cech{} persistence diagrams for arbitrary finite metric graphs across all dimensions. }

\section{Introduction}
\label{sec:introduction}

Graphs are ubiquitous in data analysis, often used to model social, biological and technological systems.
Often, data with a notion of distance can be modeled by a metric graph.
A graph is a \emph{metric graph} if each edge is assigned a positive length and if the graph is equipped with a natural metric where the distance between any two points of the graph (not necessarily vertices) is defined to be the minimum length of all paths from one to the other~\cite{Kuchment2004}.
A metric graph is therefore a special type of metric space that captures simple forms of geometric relations in data that arise in both abstract and practical settings, such as biological networks, social networks and road networks.
For example, the movement patterns that GPS systems trace for vehicles can be modeled as a metric graph for location-aware applications.
Brain functional networks as metric graphs capture the blood-oxygen-level dependent
signal correlations among different areas of the brain~\cite{BiswalYetkinHaughton1995}.
Social networks as metric graphs can encode strengths of influence between social entities (e.g.,~persons or corporations).
Extracting the topological
structures of such networks can provide powerful insights for navigating and understanding their underlying data.

Our work aims to describe topological structures of metric graphs by using persistent homology, a fundamental tool in topological data analysis that has been used in many
applications to measure and compare topological features of shapes and functions~\cite{EdelsbrunnerHarer2008}.
In this work, we give a qualitative description of information that can be captured from metric graphs using topological, persistence-based summaries.
Theorem \ref{thm:main}, the main theorem in this paper,  provides  a complete characterization of the persistence diagrams in dimension $1$ for metric graphs in a particular intrinsic setting.

\begin{thm}
\label{thm:main}
Let $G$ be a finite metric graph of genus $g$ with shortest system of loops $\{c^*_1, \ldots, c^*_g\}$, and for each $i=1,\ldots,g$,
let $|c^*_i| = \ell_i$ be the length of the $i^{th}$ loop, with $\ell_i \leq \ell_j$ for all $i \leq j$. Then the $1$-dimensional intrinsic \Cech{} persistence diagram of $G$, denoted $\Dg_1 IC_G$, consists of the following collection of points on the $y$-axis:
\[
\Dg_1IC_G =  \left\{ \left(0, \frac{\ell_i}{4}\right) : 1\leq i \leq g\right\}.
\]
\end{thm}
Terms in the statement of this main theorem, including \emph{intrinsic \Cech{} persistence diagram} and \emph{shortest system of loops}, will be rigorously defined in Sections \ref{sec:background} and \ref{subsec:update-notation}. Intuitively, this theorem provides a way to count and measure the minimal cycles in a metric graph using persistent homology. The 1-dimensional intrinsic \Cech{} persistence diagram serves as a kind of fingerprint for the graph's cycle set. Note also that the use of the word \emph{genus} is not referring to the more traditional genus of a surface in which the graph can be embedded. Instead, we use genus to mean the number of cycles in a minimal generating set of the homology of the graph. This will be made more precise in Section \ref{sec-update} when we define the shortest system of loops.

\vspace{2mm}
\noindent\textbf{Related Work.}
The work of Adamaszek et al.~\cite{AdamaszekAdams2015,AdamaszekAdamsFrick2016} is most relevant to ours, as it helps to characterize persistence diagrams in all dimensions for a metric graph consisting of a single loop.
In~\cite{AdamaszekAdamsFrick2016}, the authors show that the intrinsic Vietoris-Rips or \Cech{} complex of $n$ points in the circle $\Sspace^1$, at any scale $r$, is homotopy equivalent to either a point, an odd-dimensional sphere, or a wedge sum of spheres of the same even dimension. The results in~\cite{AdamaszekAdamsFrick2016} further imply that the $1$-dimensional homology group of a metric graph with a single loop is either rank $1$ (in the case where the associated intrinsic complex is homotopy equivalent to $\Sspace^1$) or rank $0$ (in all other cases).
One can then show that the $1$-dimensional persistence diagram consists of the
single point $\ds \left(0, \frac{\ell}{4}\right)$ or $\ds \left(0, \frac{\ell}{6}\right)$ in the case of the \Cech{} or
Vietoris-Rips filtration, respectively, where $\ell$ is the length of the loop~\cite{AdamaszekAdams2015}. (Note: here, we are assuming the convention that the Vietoris-Rips complex at scale $r$ contains all simplices of diameter at most $2r$. This definition differs by a factor of 2 from that used in~\cite{AdamaszekAdams2015}.)


In this paper, we generalize the above result in \cite{AdamaszekAdams2015}  from a metric graph with a single loop to a metric graph containing an arbitrary, finite set of loops in homological dimension 1. This characterization of persistence diagrams in dimension $1$ of an arbitrary finite metric graph complements the work in~\cite{AdamaszekAdams2015} and constitutes an important step toward the characterization of the intrinsic \Cech{} persistence diagrams of \emph{arbitrary} metric graphs across \emph{all} dimensions.

In addition to the \Cech{} and Vietoris-Rips complexes, there are a number of other types of complexes or combinatorial structures related to
graphs. 
In~\cite{Previte2014}, the author studies the relationship between properties of a graph $G$ and the homology of an associated neighborhood complex.
The paper~\cite{Taylan2016}
contains a study of so-called devoid complexes of graphs where simplices
correspond to vertex sets whose induced subgraphs do not contain certain
forbidden subgraphs. However, the neighborhood and devoid complexes are more related to structural, rather than metric, properties of graphs, so we
turn our attention in the remainder of this paper to the more metric-derived \Cech{} complex.\\


\noindent\textbf{Outline.}
The outline of the paper is as follows. In Section~\ref{sec:background}, we recall the necessary background on persistent homology, in particular for the case that the underlying topological space is a metric graph. Section~\ref{sec-update} focuses on establishing the fundamental details of the relationship between a metric graph and its associated intrinsic \Cech{} complex. We prove Theorem~\ref{thm:main} in Section~\ref{sec:proof}. Finally, we discuss our results and plans for future work in Section~\ref{sec:future}.

\section{Background}
\label{sec:background}



\subsection{Homology}
\label{subsec:hom}

Homology is an invariant that characterizes properties of a topological space $X$. In particular, the $k$-dimensional holes (connected components, loops, trapped volumes, etc.) of a space generate a homology group, $H_k(X)$. The rank of this group is referred to as the \emph{$k$-th Betti number} $\beta_k$ and counts the number of $k$-dimensional holes of $X$. We provide a brief overview of simplicial homology below. For a comprehensive study, see \cite{hatcher2002algebraic, munkres1984elements}. For a more categorical viewpoint, see \cite{rotman2009introduction}, and for a discussion of cubical complexes, see \cite{massey1991basic}.
We also note that singular homology is a related concept, and in fact is isomorphic to simplicial homology on spaces which can be triangulated \cite{hatcher2002algebraic}.
Although a priori one must work with singular homology of a metric graph $G$, rather than simplicial homology, 
metric graphs can be triangulated using a subdivision or discretization of the edges. Therefore, we will work with simplicial homology in the remainder of this paper.

A \emph{simplicial complex} $S$ is a set consisting of a finite collection of $k$-simplices, where a $k$-\emph{simplex}, given by $\sigma =  [v_0,v_1,\ldots,v_k]$, is the convex hull of distinct points $v_0,v_1,\ldots,v_k$.   Thus, a $0$-simplex is a vertex, a $1$-simplex is an edge, a $2$-simplex is a filled-in triangle, a $3$-simplex is a solid tetrahedron, and so on.
The $k$-simplices must satisfy the following: (1) if $\sigma$ is a simplex in $S$, then all lower-dimensional subsets of $\sigma$, called \emph{subsimplices}, are also in $S$; and (2) two simplices are either disjoint or intersect in a lower-dimensional simplex contained in $S$.

An algebraic structure of a  {\it vector space} or an {\it  $R$-module} over some ring $R$ is imposed on the simplicial complex $S$ to uncover the homology of the underlying topological space as follows. The $k$-simplices form a basis for a vector space, $S^{(k)}$, over some ground field (or ring) $\mathbb{F}$.  We call the vector space $S^{(k)}$ the {\it $k$-dimensional chain group} over the simplicial complex $S$. The finite field $\mathbb{Z}_p$ (where $p$ is a small prime), $\mathbb{Z}$, and $\mathbb{Q}$ are common choices for the ground field or ring. In this paper we will work over $\mathbb{Z}_2.$ Furthermore, for each pair of consecutive vector spaces there is a  linear map, $\delta_k: S^{(k)} \rightarrow S^{(k-1)}$, turning the sequence of chain groups into a {\it chain complex}:
\[\cdots \rightarrow S^{(k+1)} \xrightarrow{\delta_{k+1}} S^{(k)} \xrightarrow{\delta_{k}} S^{(k-1)} \cdots.\]
These maps are known as \emph{boundary operators}, taking each $k$-simplex to an alternating sum of its $(k-1)$-subsimplices, its boundary. More precisely, if $[v_0,v_1,\ldots,v_k]$ is a $k$-simplex, the boundary map $\delta_k: S^{(k)} \rightarrow S^{(k-1)}$ is defined by
$$\delta_k([v_0,v_1, \ldots, v_k]) = \sum_{i=0}^k (-1)^i [v_0, \ldots, \hat{v_i}, \ldots, v_k]$$
where $[v_0, \ldots, \hat{v_i}, \ldots, v_k]$ is the $(k-1)$-simplex obtained from $[v_0, \ldots, v_k]$ by removing vertex $v_i$.

The simplicial homology, $H_k(S)$,  of a simplicial complex $S$ is defined based on two subspaces of the vector space $S^{(k)}$: $Z_k=ker(\delta_k)$ known as \emph{$k$-cycles}, and $B_k=im(\delta_{k+1})=\delta_{k+1}(S^{(k+1)})$ known as \emph{$k$-boundaries}. Since the boundary operator satisfies the property $\delta_k \circ \delta_{k+1}=0$ for every $0 \leq k \leq \,$ dim$(S)$, the set of $k$-boundaries is contained in the set of $k$-cycles. Then, $H_k(S)=Z_k/B_k$ consists of the equivalence classes of $k$-cycles that are not $k+1$-boundaries (up to {\it homotopy}).
The elements of    $H_k(S)$ are called homology classes  and can thus be thought of  as equivalence classes represented by cycles enclosing $k$th order holes that differ by elements of a boundary. The rank of the associated homology group $H_k(S)$ is the number of distinct $k$ dimensional holes, and is referred to as the $k$th Betti number, denoted $\beta_k$. 

\subsection{Persistent homology and metric graphs}
\label{subsec:ph+graphs}


In \emph{persistent homology}, rather than studying the topological structure of a single space, $X$, one considers how the homology changes over an increasing sequence of subspaces.
Given a topological space $X$ and a filtration of this space,
$$X_1 \subseteq X_2 \subseteq X_3 \subseteq \ldots \subseteq X_m=X,$$ applying the homology functor gives a sequence of homology groups, called a \emph{persistence module}, induced by inclusion of the filtration
$$H_k(X_1) \rightarrow H_k(X_2) \rightarrow \ldots \rightarrow H_k(X_m).$$  The arrows above indicate the homomorphisms between homology groups induced by the filtration of the space $X$.

A filtration of a topological space $X$ may be defined in a number of ways. By considering a continuous function $f:X \rightarrow \mathbb{R}$, one may define the \emph{sublevel set filtration}
$$f^{-1}(-\infty,a_1) \subseteq  f^{-1}(-\infty,a_2)\subseteq  \ldots \subseteq  f^{-1}(-\infty,\infty).$$
Another approach is to build a sequence of simplicial complexes on a set of points using, for instance, the \emph{Vietoris-Rips filtration} \cite{citeulike:2795523} or the \emph{intrinsic \Cech{} filtration} \cite{ChazalSilvaOudot2014} discussed below.
\emph{Persistent homology} \cite{carlsson2009topology, EdelsbrunnerHarer2008} then tracks elements of each homology group through the filtration. This information may be displayed in a \emph{persistence diagram} for each homology dimension $k$. A persistence diagram is a set of points in the plane together with an infinite number of points along the diagonal where each point $(x,y)$ corresponds to a homological element that appears (is `born') at $H_k(X_x)$ and which no longer remains (`dies') at $H_k(X_y)$. See Figure \ref{barcode_PHdiag} for an example persistence diagram.
Distinct topological features may have the same birth and death coordinates; therefore, a persistence diagram is actually a multiset of points. Since all topological features die after they are born, this is an embedding into the upper half plane above the diagonal $y=x$. Points near the diagonal are often considered noise while those further from the diagonal represent more robust topological features. For a detailed description of  applications
of persistent homology to various problems in the experimental sciences, see \cite{carlsson2009topology, deSilva_Ghrist_AGT:2007, citeulike:2795523}.

\begin{figure}
  \begin{center}
    \includegraphics[width=0.4\textwidth]{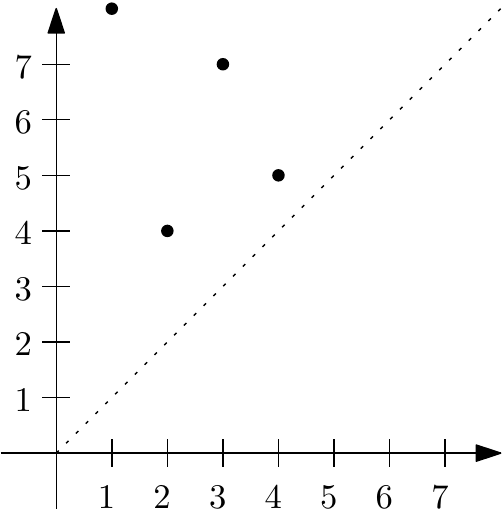}
    \caption{An example persistence diagram with four points: (1, 8), (2, 4), (3, 7), and (4, 5) corresponding to the birth and death values for distinct topological features.}
    \label{barcode_PHdiag}
  \end{center}
\end{figure}

In this paper, we focus on understanding the topological structure of a finite \emph{metric graph} in homology dimension $k=1$.
Given a graph $G=(V,E)$, we define a finite metric graph to be a metric space $(|G|,d_G)$ that is homeomorphic to a $1$-dimensional finite stratified space consisting of a finite number of $0$-dimensional pieces (i.e. vertices) and $1$-dimensional pieces (i.e. edges or loops) glued together, as described in \cite{AanjaneyaChazalChen2012, DeyShiWang2015}.
More formally, any graph $G$ with vertex set $V$ and edge set $E$, together with a length function, $length: E \rightarrow \mathbb{R}_{\geq 0}$, on $E$ that assigns lengths to edges in $E$, gives rise to a metric graph $(|G|,d_G)$ where $|G|$ is a geometric realization of $G$ and $d_G$ is defined in the following manner.
Using the notation of \cite{DeyShiWang2015}, let $e$ denote an edge in $E$ with $|e|$ being its image in $|G|,$ let $e : [0, length(e)] \rightarrow |e|$ be the arclength parametrization, and define $d_G(x,y)=|e^{-1}(y)-e^{-1}(x)|$ (with the bars indicating the absolute value) for any $x,y \in |e|$.
Now, this definition of $d_G(x,y)$ enables one to define the length of any given path between two points in $|G|$ by first restricting the path to edges in $G$ and then summing the lengths.
Then one may define the distance $d_G(u,v)$ between any pair of points $u,v \in |G|$ to be the minimum length of any path in $|G|$ between $u$ and $v$.
In Figure \ref{fig:sidebyside} we show a graph side-by-side with its corresponding metric graph.
Note that in the metric graph we have removed the vertices to emphasize that \emph{all} points along the edges are vertices.

\begin{figure}
\centering
  \includegraphics[width=0.8\linewidth]{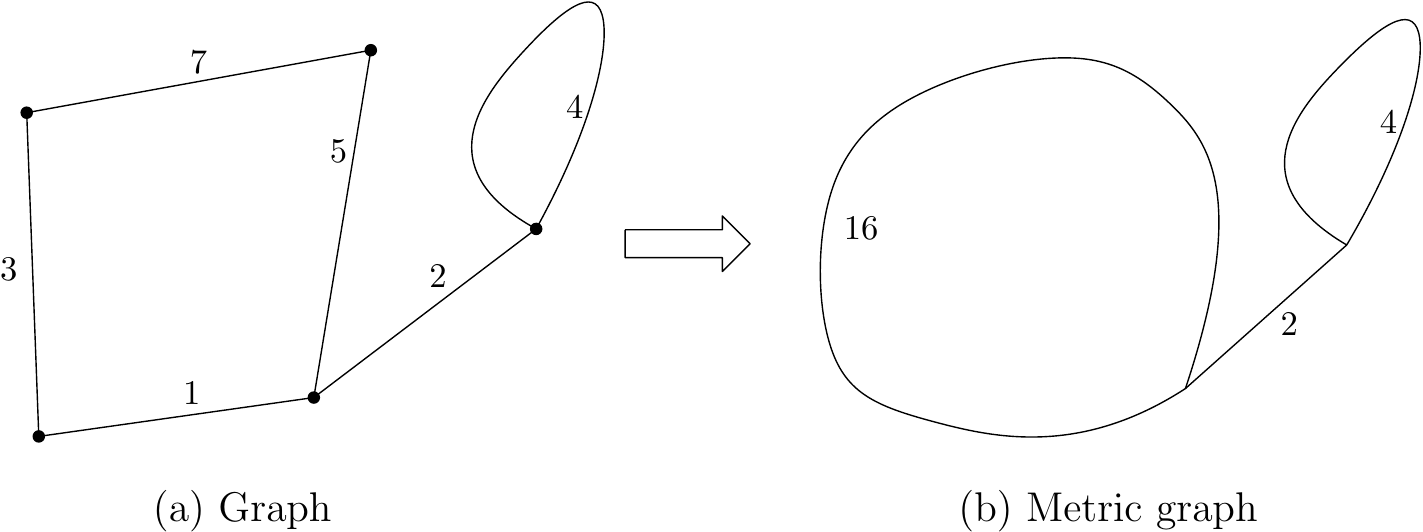}
  \caption{A traditional graph with 5 vertices and 6 edges (including one self-loop) alongside its corresponding metric graph.}
  \label{fig:sidebyside}
\end{figure}

We consider a simplicial complex built on a metric graph as follows.
Let $(G,d_G)$ be a metric graph with geometric realization $|G|$.
For any point $x\in |G|$, we define the set $B(x, \epsilon):= \{y \in |G| : d_G(x,y) < \epsilon \}$, and we let $U_\epsilon:= \{B(x,\epsilon) : x \in |G|\}$ be an open cover.
Recall that the set $|G|$ consists of all vertices and every point along an edge, an uncountable set of points.
Therefore $U_\epsilon$ is an uncountable cover.
The \emph{nerve} of a family of sets $(Y_i)_{i \in I}$ is the abstract simplicial complex defined on the vertex set $I$ by the rule that a finite set $\sigma \subseteq I$ is in the nerve if and only if $\bigcap_{i\in \sigma} Y_i \neq \emptyset$.
We let $\IC_\epsilon$ denote the nerve of $U_\epsilon$, referred to as the \emph{intrinsic \Cech{} complex}.
In Figure \ref{fig:ex_cech} we show an example subset of $U_\epsilon$ and subcomplex of $\IC_\epsilon$ for the metric graph shown in Figure \ref{fig:sidebyside}(b).
We illustrate a finite number of $B(x, \epsilon)$ intrinsic balls and their corresponding nerve.

\begin{figure}
\centering
  \includegraphics[width=\linewidth]{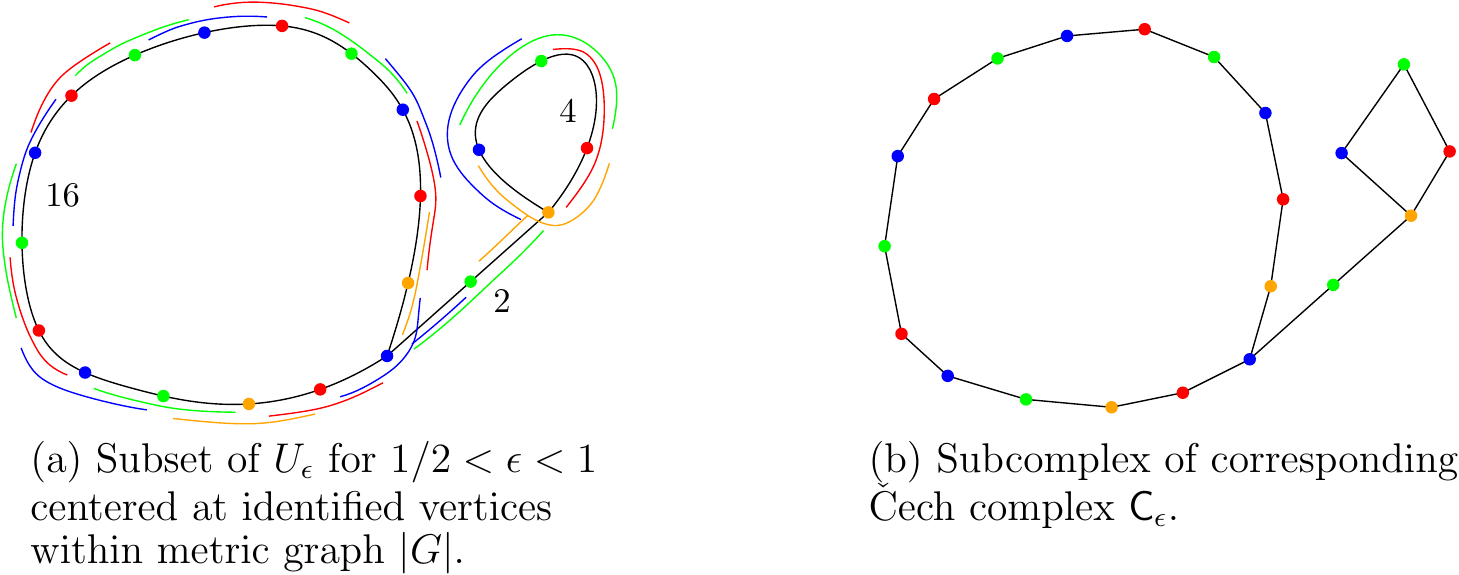}
  \caption{An example cover and corresponding nerve which we have colored to distinguish between overlapping balls and to help illustrate the correspondence between the balls  and their associated vertex in the nerve complex.}
  \label{fig:ex_cech}
\end{figure}

Let $\IC_{\epsilon}^{k}$ denote its $k$-skeleton, where there is a bijection between the vertices in $\IC_{\epsilon}^{0}$ and the points in $G$.
The associated \emph{intrinsic \Cech{} filtration} is defined as the set of inclusion
maps
\[ \{\IC_\epsilon \hookrightarrow \IC_{\epsilon'}\}_{\forall 0 \leq \epsilon \leq \epsilon'}. \]
The intrinsic \Cech{} filtration on the metric graph $G$ induces the persistence module
\[ \{H_{k}(\IC_\epsilon) \rightarrow H_{k}(\IC_{\epsilon'})\}_{\forall 0 \leq \epsilon \leq \epsilon'}\]
in any dimension $k$, from which we obtain the \emph{intrinsic \Cech{} persistence diagrams}, denoted by $\Dg_k IC_G$. In this paper, we shall only be interested in $\Dg_1IC_G$.
Finally, given an intrinsic \Cech{} complex $\IC_\delta$, we denote its $k$-dimensional chain group as $\IC_\delta^{(k)}$.

\section{From Graphs to Intrinsic \Cech{} Complexes}
\label{sec-update}

\subsection{Overview and Relevant Notations}
\label{subsec:update-notation}

We first introduce a few technical definitions and results that will enable us to prove Theorem~\ref{thm:main}. These results will provide a formal way of thinking of properties of a graph and its corresponding \Cech{} complex $\IC_\delta$, for a sufficiently small $\delta$. This will be achieved in four steps:
\begin{itemize}
\item Formal definitions of a shortest system of loops in a graph, shortest path metric on $\IC_\delta$ (Lemma~\ref{claim:equaldist}), shortest basis of $H_1(\IC_\delta)$, and $\delta$-discretization of a graph.
\item The graph $G$  has the same homotopy type as its associated intrinsic \Cech{} complex $\IC_\delta$ for a sufficiently small $\delta > 0$ (Lemma \ref{lemma:homotopy-equivalence}).
\item The inclusion of any $\delta$-discretization $\hat{G}$ of the graph $G$ into $\IC_\delta$ induces an isomorphism between their first homology groups (Lemma \ref{lemma:cycle-homotopy}). This in turn connects a shortest system of loops for $G$ to a basis for the first homology group $H_1(\IC_\delta)$ of the simplicial complex $\IC_\delta$.
\item The shortest basis of $H_1(\IC_\delta)$ corresponds to the shortest system of loops for $G$ (Lemma \ref{lemma:shortest}).
\end{itemize}

We tackle the first step in this subsection and the last three steps in Section~\ref{formalinformal}. 


\vspace{2mm}
\noindent\textbf{Shortest system of loops of $G$.}
For the graph $G$, we will initially consider its singular homology. In particular,  a singular $1$-simplex of $G$ is a continuous map $\sigma: [0,1]\to G$.
A singular $1$-chain (thus $1$-cycle) is a formal sum of such continuous maps \cite{hatcher2002algebraic}.
The objects of interest to us are actually the ``geometric representations'' of such 1-cycles, i.e. their images in $G$.
 In particular, a \emph{loop} $c$ in $G$ is a continuous map $c: \mathbb{S}^1 \to G$; we may also use ``loop'' to refer to the image of this map. (Note that we use the term ``loops'' in this manner for graphs in order to contrast with ``cycles'' in homology.)
For any singular $1$-cycle $\alpha = \sigma_1 + \cdots + \sigma_s$, there is a corresponding loop $c$ whose image in $G$ coincides with the disjoint union of images $\sigma_i([0,1])$, for $i\in [1, s]$, and we refer to $c$ as the \emph{carrier of $\alpha$}.
All singular 1-cycles carried by the same loop are homologous. 
For a genus $g$ metric graph $G$, a \emph{system of loops of $G$} refers to a set of $g$ loops $\{c_1, \ldots, c_g\}$ such that the homology classes carried by them form a basis for the $1$-dimensional (singular) homology group $H_1(G)$.  For simplicity, we say that these loops are \emph{independent}.
Given a system of loops of $G$, its \emph{length-sequence} is the sequence of lengths of elements in this set in non-decreasing order.
A system of loops forms a \emph{shortest system of loops of $G$} if its length-sequence is smallest in lexicographical order among all systems of loops of $G$.

From now on, let $\{c^*_1, \ldots, c^*_g\}$ denote a shortest system of loops of $G$ with length-sequence $\ell_1 \le \ell_2 \leq \ldots \leq \ell_g$, where $\ell_i$ is the length of the loop $c_i$ for $i = 1, \ldots, g$.

%

\vspace{2mm}
\noindent\textbf{Shortest-path distance metric on $\IC_\delta$.}
To prove Theorem \ref{thm:main}, we will only operate on intrinsic  \v{C}ech complexes. We now define a metric structure $(\IC_\delta^0, \ddelta)$ on the vertex set $\IC_\delta^0$ of the \v{C}ech complex $\IC_\delta$ for sufficiently small $\delta$.

First, we note that there is a natural bijection between points in $G$ and vertices of $\IC_\delta^0$. Specifically, for any point $x\in G$, there is a vertex $\bx$ in the nerve complex $\IC_\delta$ corresponding to the covering element ($\delta$-ball) $B(x, \delta)$.
In what follows, we say that $x$ \emph{generates} the vertex $\bx$ in the \v{C}ech complex $\IC_\delta$.

Recall $d_G$ is the shortest-path distance metric on $G$, where $d_G(x,y)$ is the length of a shortest path between points $x$ and $y$ in $G$.
Given two vertices $\bx, \by \in \IC_\delta^0$, we say that a $1$-chain $\gamma$ in $\IC_\delta^{(1)}$ \emph{connects} $\bx$ with $\by$ if $\gamma$ can be represented as $\gamma = \displaystyle \sum_{i=0}^{N} [\bv_i, \bv_{i+1}]$ such that $\bv_0 = \bx$, $\bv_{N+1} = \by$ and $[\bv_i, \bv_{i+1}]$ are edges in $\IC_\delta^{(1)}$.
For simplicity, we also refer to a 1-(simplicial) chain in $\IC_\delta$ as a \emph{path} in $\IC_\delta$.


%

\begin{definition}
\label{definition:d-ic-delta}
We define the metric $\ddelta: \IC_\delta^0 \times \IC_\delta^0 \to \mathbb{R}$ as follows.
Given two vertices $\bx, \by \in \IC_\delta^0$, if $\bx$ and $\by$ are connected by an edge in $\IC_\delta^{(1)}$, then $\ddelta(\bx, \by):=d_G(x,y)$.
Otherwise, we set $\ddelta(\bx, \by):= \inf_{\gamma} ~length(\gamma)$, where $\gamma$ ranges over all 1-chains in $\IC_\delta$ connecting $\bx$ and $\by$, and its length is defined as
$$length(\gamma) = \sum_{[\bv, \bv'] \in \gamma} \ddelta(\bv, \bv'). $$
%
\end{definition}
In other words, an elementary $1$-chain (i.e.~an edge) directly inherits the metric from the graph, and the distance between two arbitrary vertices $\bx, \by$ in $\IC_\delta^0$ is the length of the ``shortest'' $1$-chain among all $1$-chains connecting $\bx$ with $\by$.
However, note that as $\ddelta(\bx, \by)$ is defined to be the infimum of the lengths of paths connecting $\bx$ with $\by$, a priori, it is not clear that it can be realized by a shortest path from $\bx$ to $\by$. We prove the following result, whose proof also implies that there is always a realizing shortest path whose length equals $\ddelta(\bx, \by)$ for any two $\bx, \by \in \IC_\delta^0$.

\begin{lemma}\label{claim:equaldist}
For any two vertices $\bx, \by \in \IC_\delta^0$, $\ddelta (\bx, \by) = d_G(x,y)$, and the distance $\ddelta(\bx,\by)$ is realized by a shortest 1-chain connecting $\bx$ to $\by$.
\end{lemma}
\begin{proof}
We prove by contradiction.
 (i) First, we assume $\ddelta(\bx, \by) < d_G(x,y)$.
 By definition, there must exist a path $\gamma = \displaystyle \sum_{i=0}^{N} [\bv_i, \bv_{i+1}]$ in $\IC_\delta^{(1)}$ such that $\ddelta(\bx,\by)\le length(\gamma) < d_G(x,y)$.
That is,
$$\displaystyle \sum_{i=0}^{N} d_{\IC_\delta}(\bv_i, \bv_{i+1}) = \displaystyle \sum_{i=0}^{N} d_G(v_i, v_{i+1}) < d_G(x,y).$$
Therefore, the points $v_i$ in $G$ form a path connecting $x$ and $y$ of shorter length than the shortest path between $x, y$ in $G$, a contradiction.

(ii) Now assume $d_{\IC_\delta}(\bx, \by) > d_G(x,y)$. Suppose $\xi$ is a shortest path connecting $x$ and $y$ in $G$. Then, we can consider a discrete version of $\xi$, denoted as $\hat{\xi}$, such that $\hat{\xi}$ contains a finite number of vertices in $\xi$ and each edge is of length at most $\delta$. The vertices and edges in $\hat{\xi}$ give rise to a $1$-chain in
$\IC_\delta^{(1)}$ connecting $\bx$ and $\by$ that, at the same time, is shorter than $\gamma$ (the shortest $1$-chain connecting $\bx$ with $\by$). This is a contradiction.

Putting these two directions together, we have that $\ddelta(\bx, \by) = d_G(x,y)$.
%

Finally, note that by part (ii) above, any shortest path $\xi$ between $x$ and $y$ in $G$ gives rise to a $1$-chain $\hat \xi$ whose length is at most $d_G(x,y)$. Since we know that $\ddelta(\bx, \by) = d_G(x,y)$, it follows that $length(\hat \xi) = d_G(x,y)$. Hence there exists a shortest path connecting $\bx$ with $\by$ whose length realizes $\ddelta(\bx,\by)$.
\qed
\end{proof}


\vspace{2mm}
\noindent\textbf{Shortest basis of $H_1(\IC_\delta)$.}
The existence of the above metric on the vertex set of $\IC_\delta$ allows one to define the \emph{shortest basis} of $H_1(\IC_\delta)$, which relates the algebraic construction of $\IC_\delta$ to the combinatorial properties of the graph.

\begin{definition}
Given a homology class $[h] \in H_1(\IC_\delta)$, the \emph{length of $[h]$} is the shortest length of any 1-cycle of $\IC_\delta$ contained in this class. We call the corresponding minimal length cycle $\gamma$ representing $[h]$ the \emph{minimal generating cycle for $[h]$}.

The \emph{shortest homology basis of $H_1(\IC_\delta)$} consists of $\{[\gamma_i]\}_{i=1}^g$ such that the (non-decreasing) length-sequence of the basis elements is lexicographically smallest among that of all bases of $H_1(\IC_\delta)$. The set of corresponding minimal generating cycles $\{ \gamma_1, \ldots, \gamma_g \}$ is referred to as the \emph{shortest system of generating cycles for $H_1(\IC_\delta)$} (or the \emph{shortest basis of $H_1(\IC_\delta)$} for short).
\end{definition}
In Section~\ref{formalinformal}, we establish the correspondence between the graph-theoretic notion of the \emph{shortest system of loops} and the topological object \emph{shortest basis of $H_1(\IC_\delta)$} defined above.

\vspace{2mm}
\noindent\textbf{$\delta$-Discretization of $G$.}
Since we work with simplicial homology for $\IC_\delta$, we introduce a \emph{$\delta$-discretization of $G$},
denoted $\widehat{G}$, and consider the simplicial homology of $\widehat{G}$ and its relation to the simplicial homology of $\IC_\delta$, instead of the singular homology for $G$.
A $\delta$-discretization is a \emph{subdivision} of $G$, to use terminology from graph theory, with additional restrictions.
For each arc $e = [x,y] \in G$ we replace it with a path of edges $[x_0, x_1]+[x_1, x_2]+\cdots+[x_{k-1},x_k]$, where $x = x_0$, and $y = x_k$,
to obtain $\widehat{G}$ so that: (i) all graph nodes in $G$ are nodes in $\widehat{G}$;
(ii) all edges in $\widehat{G}$ are of length at most $\delta$;
and (iii) if $length(e) = \ell$, then $\ds\sum_{i=0}^{k-1} length([x_i, x_{i+1}]) = \ell$.
In Figure \ref{fig:ex_cech}(a), the chosen vertices and edges can be seen as a 1-discretization of $|G|$.
It is easy to see that $\widehat{G}$ forms a triangulation of $G$ ($G$ is the underlying space of $\widehat{G}$) and thus the (simplicial) homology of $\widehat{G}$ is isomorphic to the (singular) homology of $G$.
Furthermore, since no new loops were created, and all original arc lengths were preserved via a path of shorter edges in the subdivision, a shortest system of loops of $G$ induces a shortest system of loops of $\widehat G$ of the same lengths.
Hence from now on, we sometimes abuse the notation slightly and use $\{ c^*_1, \ldots, c_g^*\}$ to refer to a shortest system of loops in both $G$ and $\widehat G$.

\subsection{Relating Graphs to Intrinsic \Cech{} Complexes} \label{formalinformal}
In this subsection, we formalize the relations between metric graphs and their associated intrinsic  \Cech{} complexes. In particular, we provide justification for the intuition that the shortest loops in a graph correspond to the ``smallest'' basis for the first persistent homology, i.e., shortest loops can be thought of as shortest cycles.

\begin{lemma}
\label{lemma:homotopy-equivalence}
Let $\delta <  \frac{\ell_1}{4}$, where $\ell_1$ is the length of the shortest loop in the shortest system of loops for the graph $G$. Then $\IC_\delta$ is homotopy equivalent to $G$.
\end{lemma}

\begin{proof}
We will use the nerve lemma \cite{hatcher2002algebraic} which states that the nerve of a \emph{good cover} of a space $X$ is homotopy equivalent to the space $X$.
A good cover is one in which the intersection of any finite collection of sets in the cover is either empty or contractible.
We will prove by contradiction that for a sufficiently small $\delta$, the cover $U_\delta$ is a good cover.
We first show that any intersection must contain one connected component, and then show that this component must be contractible.

Assume by way of contradiction, $\hat{U}_{\delta} = \{B(x_1, \delta), \ldots, B(x_k, \delta)\}$ is some finite collection of sets in $U_\delta$ such that $B(x_1, \delta) \cap B(x_2, \delta) \cap \cdots \cap B(x_k, \delta)$ contains at least two connected components.
Let $U$ and $V$ be two of the connected components contained in this intersection, and consider points $p \in U$ and $q \in V$.
Since $p$ and $q$ are in all $B(x_i, \delta)$, there must be a path, $\pi_i$, within $|G|$ from $p$ to $q$ in each $B(x_i, \delta)$.  Notice that we can choose the $\pi_i$ to be simple paths such that the length of each $\pi_i$ is less than $2\delta$.
Additionally, there must be at least two of these paths which are not identical.
If all paths are identically equal then this path would have to be contained in the intersection, contradicting the fact that $p$ and $q$ are in disjoint connected components within the intersection.
Without loss of generality, say $\pi_1$ in $B(x_1, \delta)$ and $\pi_2$ in $B(x_2, \delta)$ are the two distinct paths.

It is possible that these two paths are not entirely disjoint, but because they are not identical, we can find a portion of each which is unique to that path.
Travel along both $\pi_1$ and $\pi_2$ from $p$ until the paths diverge for the first time at point $p' \in |G|$.
Then, continue to travel along both until the paths join back again for the first time at point $q' \in |G|$.
Let $\pi_1'$ be the portion of $\pi_1$ between $p'$ and $q'$, and similarly $\pi_2'$ is the portion of $\pi_2$ between $p'$ and $q'$.
Now, it is clear that $\pi_1'$ and $\pi_2'$ are disjoint except for their endpoints.
See Figure \ref{fig:goodcovercontradiction} for an illustration.
Moreover, the lengths of $\pi_1'$ and $\pi_2'$ must be less than $2\delta$ since they are paths contained within two $\delta$ balls.
Notice that we have created a cycle by taking $\pi_1'$ from $p'$ to $q'$ and then $\pi_2'$ from $q'$ back to $p'$.
This cycle has length $length(\pi_1') + length(\pi_2') < 2\cdot 2\delta < \ell_1$.
But, this is a contradiction because we assumed that $\ell_1$ is the length of the shortest loop in the shortest system of loops of G, and we have discovered a cycle with length less than $\ell_1$.  Thus, for any finite collection of sets $\hat{U}_{\delta}$ in the cover $U_\delta$, the intersection must only contain one connected component.

Now, we must show this component is contractible.  Since our underlying space is a metric graph, notice that a connected component must either be a metric tree, or must contain one or more loops.  A metric tree is simply connected, hence contractible, so there is nothing to prove in this case.  We claim that the other case, where the component contains a loop, cannot happen.  If the component did contain a loop, the loop would then be contained entirely inside a $\delta$-ball, and thus would have length less than $2\delta$.  But, by our choice of $\delta$, this is impossible since $2\delta$ is smaller than the length of the shortest loop.  We can therefore conclude that the component is contractible, and hence, $U_\delta$ is a good cover.  By the nerve lemma, it follows that $\IC_\delta$, the nerve of $U_\delta$, is homotopy equivalent to $G$.
\qed
\end{proof}

\begin{figure}
\begin{center}
\includegraphics[width = 0.65\textwidth]{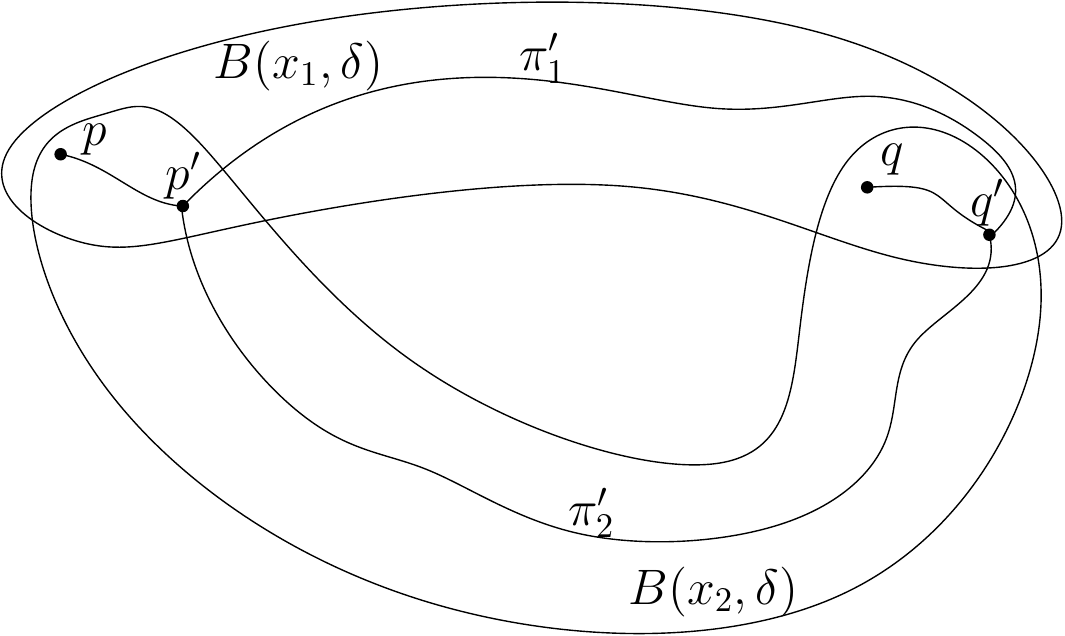}
\caption{An illustration of the points and paths described in Lemma \ref{lemma:homotopy-equivalence}.}
\label{fig:goodcovercontradiction}
\end{center}
\end{figure}

Before we state and prove the next result, note that there is a natural inclusion map $\iota: \widehat{G} \hookrightarrow \IC_{\delta}$. (Note: this map is defined on $\widehat{G}$ rather than $G$, as it would not be continuous otherwise.) Indeed, each vertex $x \in \widehat{G}$, as we described earlier, is identified with a unique vertex $\iota(x):=\bx  \in \IC_\delta^0$ corresponding to the covering element $B(x, \delta)$.
Given an edge $[x,y] \in \widehat{G}$, as $d_G(x,y) \le \delta$ (by construction of the $\delta$-discretization), $B(x, \delta) \cap B(y, \delta) \neq \emptyset$. This means that $[\bx, \by] \in \IC_\delta$.

Recall that $G$ is isomorphic to the underlying space of $\widehat{G}$; for simplicity, we identify $G$ with the underlying space of $\widehat{G}$.
Given any vertex $\bx \in \IC_\delta^0$, the point $x\in G$ that generates it may not be a vertex in $\widehat{G}$, in which case there is a unique edge $\sigma \in \widehat{G}$ whose underlying space $|\sigma| \subset G$ contains $x$.
We say that the edge $\iota(\sigma) \in \IC_\delta^{(1)}$ \emph{covers} $\bx$ (resp. $\sigma$ \emph{covers} $x$) in this case.

Finally, given a 1-chain $\gamma$ of $\IC_\delta$ with only one connected component, we can write it in the form of an ordered sequence of edges: $\gamma = [\bx_1, \bx_2] + [\bx_2, \bx_3] +\cdots + [\bx_{k-1}, \bx_{k}]$, where each $[\bx_i, \bx_{i+1}]$ is an edge (1-simplex) in $\IC_\delta$. Note that it is possible that $\bx_i = \bx_j$ for some $i \neq j$.
To emphasize this representation of the 1-chain, or path, $\gamma$, we also represent it by an ordered sequence of vertices $\langle \bx_1, \bx_2, \ldots, \bx_k\rangle$. A \emph{sub-path} of $\gamma$ is simply the 1-chain represented by a subsequence $\langle \bx_i, \bx_{i+1}, \ldots, \bx_j \rangle$.
For a cycle $\gamma$ satisfying  $\bx_k = \bx_1$, a sub-path will be represented by a subsequence from the cyclic sequence $\langle \bx_1, \bx_2, \ldots, \bx_{k-1}, \bx_1 \rangle$. For example, the subsequence $\langle \bx_{k-2}, \bx_{k-1}, \bx_1, \bx_2, \bx_3 \rangle$ represents a sub-path of the cycle $\gamma$.

\begin{lemma}
 \label{lemma:cycle-homotopy}
Assume $\delta < \ell_1/4$. The inclusion $\iota: \widehat{G} \hookrightarrow \IC_{\delta}$ induces an isomorphism $\iota_*: H_1(\widehat{G}) \rightarrow H_1(\IC_\delta)$.
\end{lemma}

\begin{proof}
For $\delta < \ell_1/4$, we know from Lemma~\ref{lemma:homotopy-equivalence} that $H_1(G)$ and $H_1(\IC_\delta)$ are isomorphic, implying that $\rank(H_1(\widehat{G})) = \rank(H_1(\IC_\delta))$. Thus, if we can show that the mapping $\iota_*: H_1(\widehat{G}) \rightarrow H_1(\IC_\delta)$ is surjective, then the inclusion $\iota$ must induce an isomorphism between these groups.
In the following, we prove that any $1$-cycle $\gamma$ in $\IC_\delta$ is homologous to the image of some cycle in $\widehat{G}$ under the inclusion $\iota$.

Without loss of generality, we assume that $\gamma$ has only one connected component (if it has more than one, we apply the same argument to each of its components).
We view $\gamma$ as a path $\gamma = [\bx_1, \bx_2] + [\bx_2, \bx_3] + \ldots  + [\bx_s, \bx_1]$.

In order to prove the lemma, we require the following two technical sublemmas. In them, we take a divide-and-conquer approach by decomposing $1$-chains in $\IC_\delta$ and the images of $1$-chains in $\widehat{G}$ under $\iota$ into pieces, and discuss the relations among these pieces. 

\vspace{2mm}
\noindent\textbf{Sublemma 1}: \emph{Given an edge $\alpha' = [\by, \by']$ in $\iota(\widehat G)$, where
$$\alpha = [\bx_j, \bx_{j+1}] + [\bx_{j+1}, \bx_{j+2}] + \cdots + [\bx_{k-1}, \bx_k]$$ 
is a path in $\IC_\delta$ such that $\alpha'$ covers $\bx_j, \bx_{j+1}, \ldots , \bx_k$, we have
\begin{equation}
\label{eq:remark2}
\alpha' - \alpha  = [\bx_j, \by] + [\bx_k, \by'] + \bdr [\bx_k,\by, \by']
+ \bdr \sum_{i=j}^{k-1} [\bx_i, \bx_{i+1}, \by].
\end{equation}
The triangles on the right hand side of \eqref{eq:remark2} belong to $\IC_\delta$, so that $\alpha' $ is homologous to $\alpha + [\bx_j, \by] + [\bx_k, \by']$.
}


\begin{figure}[!ht]
\begin{center}
\includegraphics[width = 0.7\textwidth]{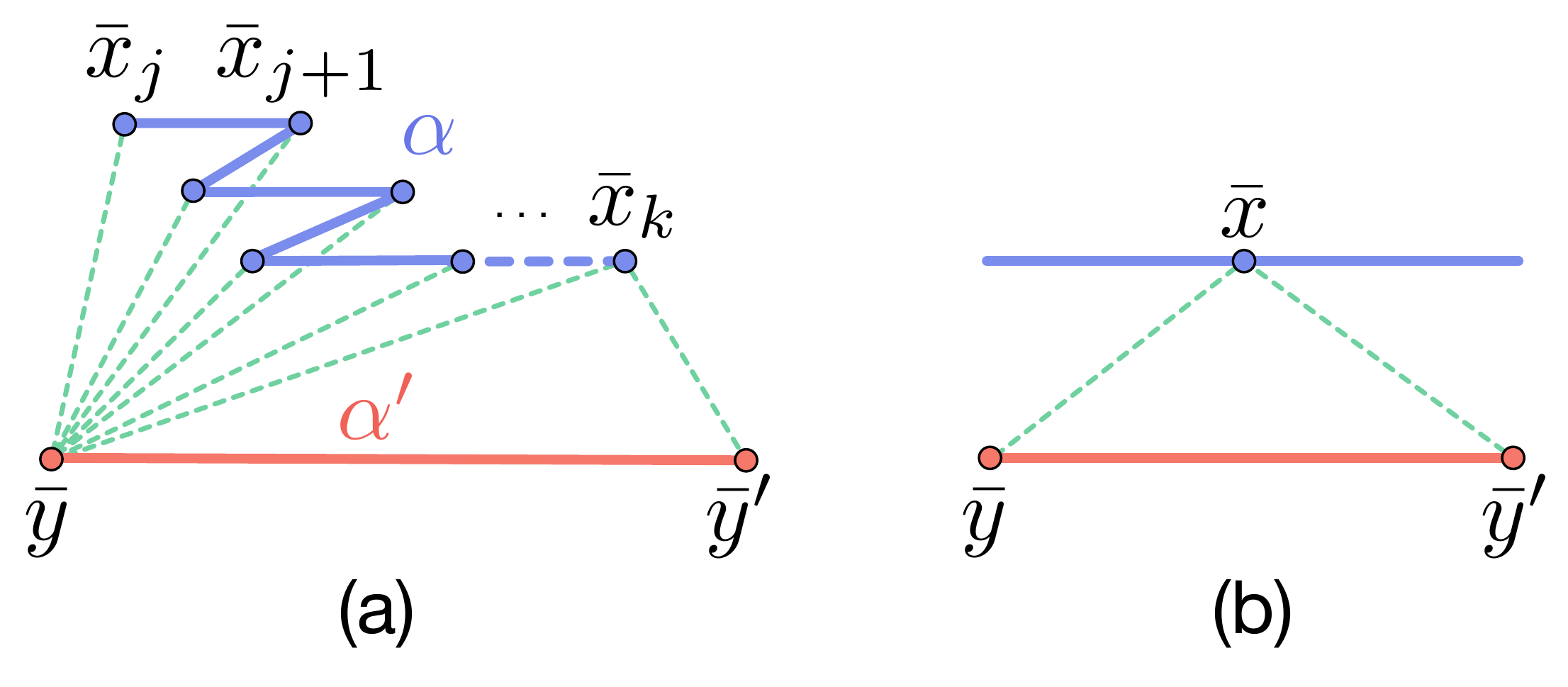}
\caption{Illustration for Sublemma 1: the (a) general and (b) degenerate cases, respectively. The blue edges and points are in $\IC_\delta$ but not $\widehat{G}$, and the red are in both.}
\label{fig:case-a}
\end{center}
\end{figure}

\begin{proof}[Sublemma 1]
See the illustration in Figure~\ref{fig:case-a}.
Since each $x_i$ is covered by the edge $[y, y'] \in \widehat{G}$, it follows that $d_G(x_i, y), d_G(x_i, y'), d_G(x_i, x_{i+1}) \le \delta$ for any $i \in [j, k]$. It then follows that $B(x_k, \delta) \cap B(y, \delta) \cap B(y', \delta) \neq \emptyset$ and $B(x_i,\delta)\cap B(x_{i+1}, \delta) \cap B(y, \delta) \neq \emptyset$. Hence the triangles $[\bx_k,\by, \by']$ and $[\bx_i, \bx_{i+1}, \by]$ all belong to $\IC_\delta$. Note that this holds for the degenerate case as well, where $\alpha'$ is homologous to $[\bx, \by] + [\bx, \by']$.
\qed
\end{proof}

\vspace{2mm}
\noindent\textbf{Sublemma 2}: \emph{
Given an elementary $1$-chain $\alpha = [\bx, \bx']$ in $\IC_\delta$,
consider the shortest path $\xi$ in $G$ connecting $x$ and $x'$. Assume that $\xi$ contains a non-empty set of vertices $y_1, y_2, \cdots, y_l$ in $\widehat{G}$, ordered by their positions in $\xi$.
Consider a $1$-chain $\alpha'$ in $\IC_\delta$ of the form
$$\alpha' = [\by_1, \by_2] + [\by_2, \by_3] + \cdots + [\by_{l-1}, \by_l].$$
Obviously, $\alpha'$ is contained in the image $\iota(\widehat{G})$ of $\widehat G$ in $\IC_\delta$, as each $[y_i, y_{i+1}] \in \widehat G$ for $i\in [1, l]$. Then 
\begin{equation}
\label{eq:remark3}
\alpha - \alpha'  = [\bx, \by_1] + [\bx', \by_l] + \bdr [\by_l,\bx, \bx']
+ \bdr \sum_{i=1}^{k-1} [\by_i, \by_{i+1}, \bx],
\end{equation}
where all the triangles on the right hand side belong to $\IC_\delta$. This implies that $\alpha' - \alpha$ is homologous to $[\bx, \by_1] + [\bx', \by_l]$, and thus $\alpha$  is homologous to $\alpha'+ [\bx, \by_1] + [\bx', \by_l]$.
}

\begin{figure}[!ht]
\begin{center}
\includegraphics[width = 0.7\textwidth]{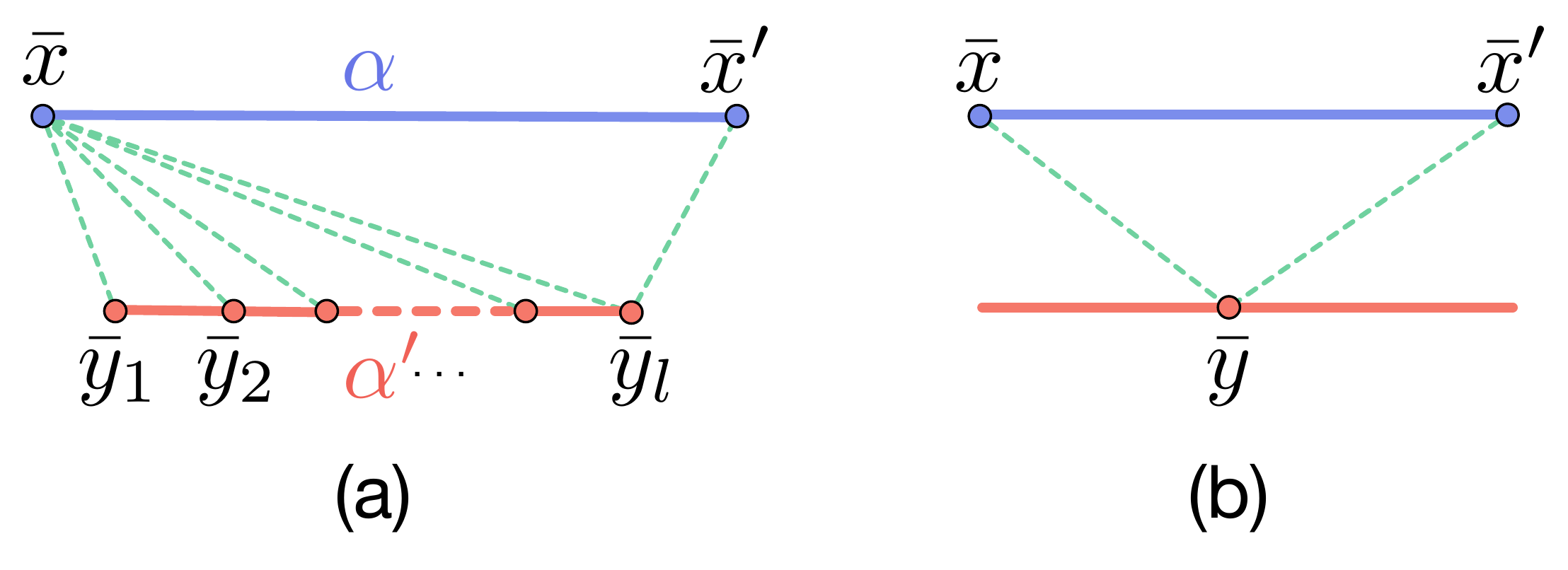}
\caption{Illustration of the (a) general and (b) degenerate cases in Sublemma 2. As before, the blue edges and points are in $\IC_\delta$ but not $\widehat{G}$, and the red are in both.}
\label{fig:case-b}
\end{center}
\end{figure}

\begin{proof}[Sublemma 2]
See the  illustration in Figure~\ref{fig:case-b}. Note that since $[\bx, \bx']$ is an edge in $\IC_\delta$, $length(\xi) = d_G(x, x') \le 2\delta$. Let $z$ be the mid-point of $\xi$; obviously, we know $d_G(x, z), d_G(x', z) \le \delta$.
Since each $y_i \in \xi$, we have $d_G(y_i, z) \le \delta$ as well for $i\in [1, l]$. It then follows that $z \in B(x, \delta) \cap B(y_i, \delta) \cap B(x', \delta) \neq \emptyset$ and $z \in B(y_i, \delta) \cap B(y_{i+1}, \delta) \cap B(x, \delta) \neq \emptyset$.
Hence  the triangles $[\by_l,\bx, \bx']$ and $[\by_i, \by_{i+1}, \bx]$ all belong to $\IC_\delta$.
Note that the argument holds for the degenerate case shown in Figure~\ref{fig:case-b} (b) in which case we assume that $[\by, \by]$ is a degenerate 1-chain to simplify the presentation of the argument.
\qed
\end{proof}

\begin{figure}[!ht]
\begin{center}
\includegraphics[width = 0.7\textwidth]{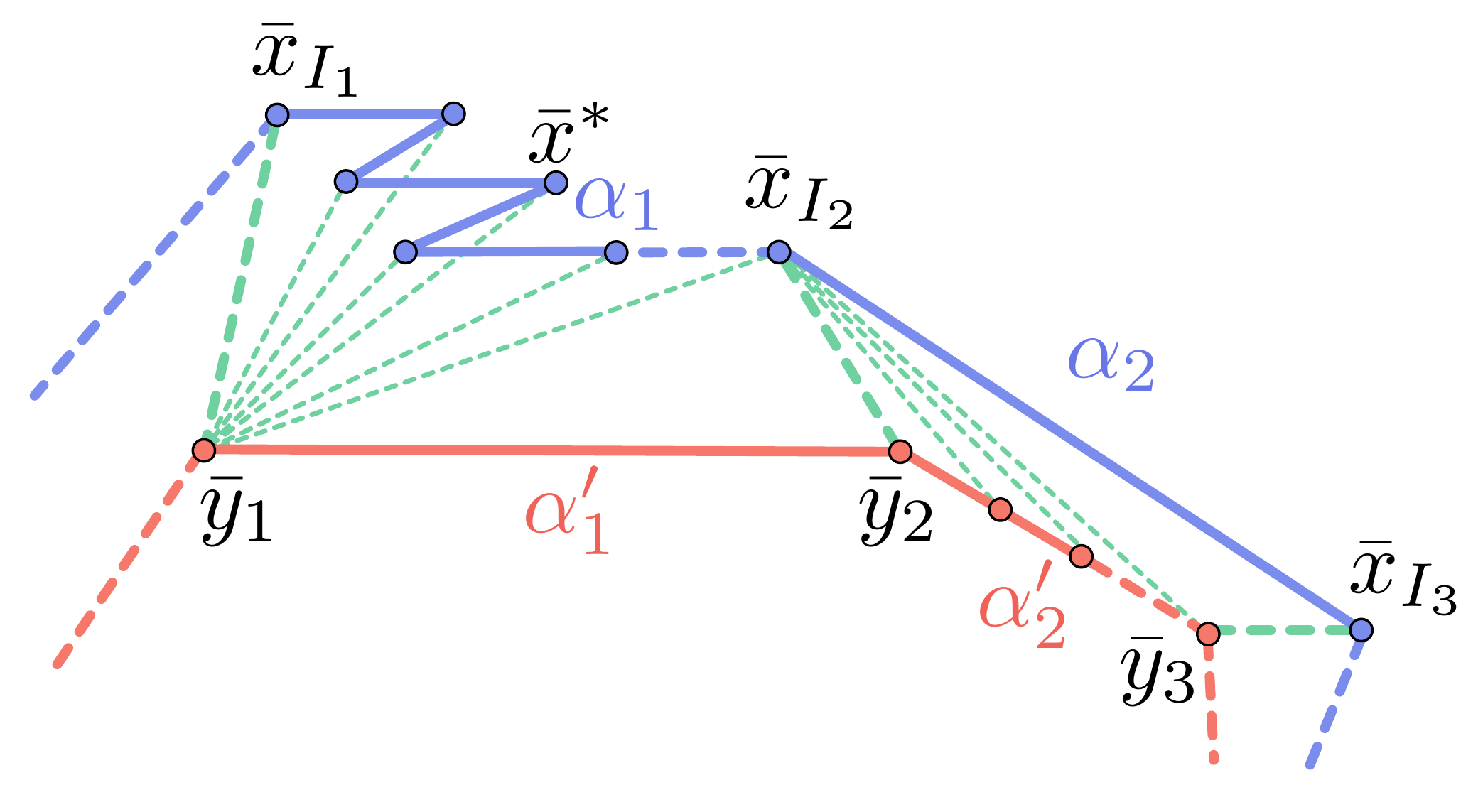}
\caption{Illustration for Lemma~\ref{lemma:cycle-homotopy}.}
\label{fig:case-c}
\end{center}
\end{figure}

We now prove our Lemma~\ref{lemma:cycle-homotopy} using the previous two sublemmas.
The proof is illustrated in Figure~\ref{fig:case-c}.
Given the 1-cycle $\gamma = [\bx_1,\bx_2] + [\bx_2+\bx_3] + \cdots + [\bx_s, \bx_1]$ in $\IC_\delta$ as described earlier, we start with a vertex $\bx^*$ in $\gamma$ (say $\bx^* = \bx_1$).
Without loss of generality, suppose $\bx^*$ is contained in some edge $\alpha_1' = [\by_{1}, \by_{2}] \in \iota(\widehat G)$. We extend in the forward and backward direction along $\gamma$ to find the maximum sub-path $\alpha_{1}$ of $\gamma$ whose vertices are contained in the edge $\alpha_{1}'$. (This is the situation described in Sublemma 1.)
Denote the starting vertex of $\alpha_{1}$ as $\bx_{I_1}$ and the ending vertex as $\bx_{I_2}$.
Now consider the next edge $\alpha_{2} = [\bx_{I_2}, \bx_{I_3}]$ in $\gamma$ where $I_3 \equiv I_{2 + 1 \bmod s}$.
By construction of $\alpha_1$, we know that $x_{I_2}$ and $x_{I_3}$ are necessarily covered by different edges in $\widehat G$. We then construct a path $\alpha_{2}' \subset \iota(\widehat{G})$ by the procedure described in Sublemma 2. Suppose $\alpha_2'$ is represented by the ordered sequence of vertices $\by'_1, \ldots, \by'_l$. Note that it is necessary that $\by'_1 = \by_2$ (which is one of the endpoints of edge $\alpha_1'$).


We then repeat this process. As we traverse $\gamma$, we alternate between the situations in Sublemma 1 and Sublemma 2.
Therefore, the $1$-cycle $\gamma$ can be represented as a linear combination of $1$-chains
$$\gamma = \sum_{i}^k \alpha_i. $$
Set
$\gamma' =\ds \sum_{i} \alpha_i'.$ We argue that $\gamma'$ is a 1-cycle. Indeed, by construction, the  last vertex in the path representation of $\alpha_i'$ is necessarily the same as the first vertex in the path representation of $\alpha_{i+1~mod~k}$. It then follows that
\begin{equation}
\label{eq:lemma2}
\gamma' - \gamma  = \sum_{i} (\alpha'_i - \alpha_i).
\end{equation}

Now, combining the claims in Sublemmas 1 and 2, we see that all of the single edges when moved to the right-hand sides of equations~\eqref{eq:remark2} and ~\eqref{eq:remark3} are cancelled out (under $\mathbb{Z}_2$ coefficients) in equation~\eqref{eq:lemma2}, leaving only triangles.
This is illustrated in Figure~\ref{fig:case-c}, where all edges $[\bx_i, \by_i]$  cancel while the triangles remain. All of the remaining triangles belong to $\IC_\delta$.
Hence $\gamma' - \gamma$ bounds a 2-chain, and  thus $\gamma$ and $\gamma'$ are homologous.
\qed
\end{proof}

\begin{lemma}\label{lemma:shortest}
Let $\{[\gamma_i]\}_{i=1}^g$ be a shortest homology basis for $H_1(\IC_\delta)$ with $\{ \gamma_i \}_{i=1}^g$ being the corresponding shortest system of generating cycles, sorted in non-decreasing order of their lengths.
Then, its length-sequence equals that of the shortest system of loops of $G$.
\end{lemma}
\begin{proof}
As before, let $\{c_1^*, \ldots, c_g^*\}$ be a shortest system of loops of $\widehat{G}$ with length sequence $\ell_1 \le \ell_2 \le \ldots \le \ell_g$.
Note that this is the same length-sequence of a shortest system of loops for $G$, as $\widehat G$ is a triangulation of $G$.

First, note that by Lemma \ref{lemma:cycle-homotopy}, $\{[\iota(c^*_i)] \}_{i=1}^g$ form a basis for $H_1(\IC_\delta)$. By applying Lemma~\ref{claim:equaldist} to each edge in $c^*_i$, we see that $length(\iota(c^*_i)) = length(c^*_i)$, and thus we know that the length-sequence induced by $\{[\gamma_i]\}_{i=1}^g$ must be lexicographically smaller than or equal to $\{ \ell_1, \ell_2, \ldots, \ell_g\}$.

Next, we prove the other direction. Specifically, set $\ell_i' = length(\gamma_i)$ for $i\in [1, g]$. We will prove by contradiction that the length-sequence $\{ \ell_1', \ell_2', \ldots, \ell_g' \}$ is lexicographically larger than or equal to $\{ \ell_1, \ell_2, \ldots, \ell_g\}$.
Assume otherwise; that is, $\{ \ell_1', \ell_2', \ldots, \ell_g' \}$ is lexicographically smaller than $\{ \ell_1, \ell_2, \ldots, \ell_g\}$.

Let $\bar{V}_\Gamma$ denote the set of vertices from all of the $\gamma_i$, and let $V_\Gamma = \{ v \in G \mid \bv \in \bar{V}_\Gamma\}$ be the set of corresponding generating points in $G$.
This set is finite since each 1-chain $\gamma_i$ is finite by definition, and there are finitely many generators.
We first refine $\widehat{G}$ to $\newGhat$ by subdividing edges of $\widehat G$ so that all points in $V_\Gamma$ are now also vertices in $\newGhat$. Obviously, $\newGhat$ is also a $\delta$-discretization of $G$, and thus Lemma \ref{lemma:cycle-homotopy} holds for it as well.

For each $i\in [1,g]$, suppose the 1-cycle $\gamma_i$ has the form $\ds\sum_{j=1}^k [\bx_j, \bx_{j+1}]$ with $\bx_{k+1}=\bx_1$. It is easy to see that $\gamma_i$ has only one component. If it has more than one, then there exists at least one component whose corresponding homology class is independent of $[\gamma_1], \ldots, [\gamma_{i-1}]$. We can set $\gamma_i$ to be this component (which is a 1-cycle itself) and obtain a shorter length-sequence, which contradicts the assumption that $\{[\gamma_i]\}$ is a shortest homology basis for $H_1(\IC_\delta)$.

We now construct a 1-cycle $\xi_i$ from $\newGhat$ satisfying the following conditions: (C-1) $length(\xi_i) = length(\gamma_i)$; and (C-2) $[\iota(\xi_i)] = [\gamma_i]$ where $\iota: \newGhat \hookrightarrow \IC_\delta$ is the inclusion map.
First, note that by construction of $\newGhat$, for each $\bx_j$, there is a vertex $x_j$ from $\newGhat$ such that $\iota(x_j) = \bx_j$.
For each $[\bx_j, \bx_{j+1}]$, consider the shortest path (1-chain) $\pi_j$ connecting $x_j$ to $x_{j+1}$ in $\newGhat$. We have that
$$length(\pi_j) = d_G(x_j, x_{j+1}) = \ddelta(\bx_j, \bx_{j+1}) = length([\bx_j, \bx_{j+1}].$$
Concatenating all such shortest chains gives rise to a 1-cycle $\xi_i$ in $\newGhat$ whose length equals $\ds\sum_{j=1}^k length([\bx_j, \bx_{j+1}]) = length(\gamma_i)$.
Hence, condition (C-1) above is satisfied.

Furthermore, recall the 1-chain $\pi_j$ in $\newGhat$ corresponding to edge $[\bx_j, \bx_{j+1}]$ constructed above. Suppose $\pi_j$ is represented by the ordered sequence of vertices $\langle x_j = y_1, y_2, \ldots, y_s = x_{j+1}\rangle$. We claim that the 1-chain $\iota(\pi_j)$ is homotopic to the edge $[\bx_j, \bx_{j+1}]$ in $\IC_\delta$. Indeed, all vertices from $\pi_j$ are contained in a path of length $d_G(x_j, x_{j+1}) \le 2\delta$ (as $[\bx_j, \bx_{j+1}]$ is an edge in $\IC_\delta$). It then follows that $B(x_j, \delta) \cap B(y_a, \delta) \cap B(y_{a+1}, \delta) \neq \emptyset$ for any $a\in [1, s-1]$ (the common intersection contains, say, the mid-point of path $\pi_j$). Hence, triangles of the form $[\bx_j, \by_a, \by_{a+1}]$, for $a\in [1,s-1]$, exist in $\IC_\delta$, establishing a homotopy between $\iota(\pi_j) = \langle \bx_j, \by_2, \ldots, \by_s = \bx_{i+1}\rangle$ and the edge $[\bx_j, \bx_{j+1}]$ in $\IC_\delta$. All of these homotopies together, for all $i\in [1,k]$, provide a homotopy between the 1-cycle $\iota(\xi_i)$ and $\gamma_i$. Thus, $[\iota(\xi_i)] = [\gamma_i]$ and condition (C-2) above also holds.

Hence, from $\{ \gamma_i \}_{i=1}^g$, we can obtain a set of 1-cycles $\{ \xi_i \}_{i=1}^g$ of $\tilde G$ whose length-sequence equals $\{ \ell'_1, \ldots, \ell'_g\}$.
Furthermore, by Lemma \ref{lemma:cycle-homotopy} and condition (C-2) above, $\{ [\xi_i] \}_{i=1}^g$ must form a basis for $H_1(\newGhat)$ as $\{ [\iota(\xi_i)] \}_{i=1}^g = \{ [\iota(\xi_i)] \}_{i=1}^g$ form a basis for $H_1(\IC_\delta)$.
It then follows that we have obtained a system of loops $\{ \xi_i \}_{i=1}^g$ for $\newGhat$ whose length-sequence is smaller than $\{ \ell_1, \ldots, \ell_k \}$, which contradicts the fact that the latter is the shortest length-sequence possible.
Hence, our assumption is wrong, and the length-sequence for $\{ \gamma_i \}_{i=1}^g$ cannot be smaller than $\{ \ell_1, \ldots, \ell_k \}$.

Therefore, putting the proofs of both directions together, we have that the length-sequence for the shortest homology basis $\{ [\gamma_i] \}_{i=1}^g$ must be equal to that of the shortest system of loops for $\newGhat$ and thus, for $G$.
\qed
\end{proof}








\section{Proof of Main Theorem}
\label{sec:proof}

We are now ready to prove Theorem~\ref{thm:main}. Let $\mu_\epsilon : \IC_\delta \rightarrow \IC_\epsilon$ for $\epsilon > \delta$ denote the chain map given by inclusion, and let $\mu_\epsilon^c : \IC_\delta^{(1)} \rightarrow \IC_\epsilon^{(1)}$ denote the associated inclusion map of one dimensional chain groups. The latter induces the map on one dimensional homology $\mu_\epsilon^h : H_1(\IC_\delta)\rightarrow H_1(\IC_\epsilon)$. Let $\gamma \in \IC_{\epsilon}^{(1)}$ denote a cycle, with $[ \gamma] \in H_1(\IC_\epsilon)$  the corresponding homology class.

\begin{proof}[Proof of Theorem~\ref{thm:main}]
Let $\{[\gamma_i]\}_{i=1}^{g}$ be a shortest homology basis for $H_1(\IC_\delta)$ corresponding to the shortest system of loops $\{c_1^*,\ldots,c_g^*\}$ of $G$. First, note that $\mu_\epsilon^c(\gamma_i)$ is a boundary cycle in $\IC_\epsilon$ for $\ds \epsilon=\frac{\ell_i}{4}$. 
This is due to the fact that, for any triple of points $x,y,z \in \gamma_i$, $\ds B\left(x,\frac{\ell_i}{4}\right) \bigcap B\left(y,\frac{\ell_i}{4}\right) \bigcap B\left(z,\frac{\ell_i}{4}\right) \neq \emptyset$.  Therefore $\gamma_i$ must die at $\ds \frac{\ell_i}{4}$ or earlier.  The rest of the proof consists of showing that:
\begin{itemize}
\item[A)] For $i=1,\ldots,g$, the $i^{th}$ cycle does not die before $\ds \epsilon= \frac{\ell_i}{4}$; and
\item[B)] No other cycles are created due to interference between cycles.
\end{itemize}

Notice that A) and B) can be reformulated to the language of bases, where condition A) is equivalent to a linear independence condition and B) is equivalent to a spanning condition. Therefore, the proof of Theorem~\ref{thm:main} follows from Proposition~\ref{prop:main} below.
\qed
\end{proof}

\begin{prop}
\label{prop:main}
For any $i=1,\ldots,g$, the set \[\{
[\mu_\epsilon^c(\gamma_i)], [\mu_\epsilon^c(\gamma_{i+1})],\ldots,
[\mu_\epsilon^c(\gamma_g)]\}\] is a basis for $H_1(\IC_\epsilon)$ where
$\ds\frac{\ell_{i-1}}{4} \leq \epsilon < \frac{\ell_i}{4}$ and $\ell_0= 0$.
\end{prop}

\begin{proof}[Proposition~\ref{prop:main}]
We will prove the two conditions A) and B).

For A), we show that $\ds\sum_{j=i}^g c_j [\mu_\epsilon^c(\gamma_j)] = [0]$
implies $c_j=0$ for all $j$. Let $\gamma = \ds\sum_{j=i}^g c_j
\mu_\epsilon^c(\gamma_j)$ be a cycle representing the trivial class $[0]=[\gamma] \in H_1(\IC_\epsilon)$.
Assume, by way of contradiction, that there exists $j$ with $i\leq j\leq g$ such that $c_j \neq 0$. Since $[\gamma]=[0],$
 there exists a $2$-dimensional chain $\alpha \in \IC_\epsilon$ having $\gamma$ as its boundary, i.e., $\partial
\alpha = \gamma$. Let $\ds \alpha = \sum_{k \in J} \Delta_k$ where, for some index set $J$, $\{\Delta_k\}_{k \in J}$ is the set of $2$-simplices in the triangulation of $\alpha$, and where for each $k$,
$t_k:=\partial \Delta_k \in \IC_\epsilon^{(1)}.$
Then $\ds \gamma = \partial \alpha = \partial \sum_{k} \Delta_k = \sum_k \partial \Delta_k = \sum_k t_k$, i.e.
\begin{equation}
\label{eq:gamma}
\gamma= \ds \sum_{j=i}^g c_j \mu_\epsilon^c(\gamma_j) = \sum_k t_k.
\end{equation}

We aim to contradict the fact that some $c_j\neq 0$ in the above sum. 

To this end, we define a map $\rho:\IC_\epsilon^{(1)}\rightarrow \IC_\delta^{(1)}$ for $\epsilon > \delta$ by specifying its effect on 
 edges in the \Cech{} complex $\IC_\epsilon$ and  extending the map linearly to all $1$--chains in  $\IC_\epsilon^{(1)}.$
First, there is a natural bijection between the set of vertices of $\IC_\delta$ and that of $\IC_\epsilon$; specifically, the vertex in $\IC_\delta$ representing the $\delta$-ball $B(u, \delta)$ corresponds to the vertex in $\IC_\epsilon$ representing the covering element $B(u, \epsilon)$. 
For simplicity, we assume $\IC_\delta^0 = \IC_\epsilon^0$ and use $\bar{u}$ to denote the vertex in $\IC_\delta$ (resp. in $\IC_\epsilon$) representing the covering element $B(u, \delta)$ (resp. $B(u, \epsilon)$). 
Now, given an edge $e = [\bar u, \bar v] \in \IC_\epsilon^{(1)}$, we describe its image $\rho(e)$ in $\IC_\delta^{(1)}$. 
If $[\bar u,\bar v]$ spans an edge in $\IC_\delta$, then we set $\rho(e)$ to be that edge. 
Otherwise, the existence of the edge $[\bar u,\bar v] \in \IC_\epsilon^{(1)}$ implies, by definition, that $B(u,\epsilon) \cap B(v,\epsilon) \neq
\emptyset$. 
Choose an arbitrary (but fixed with respect to $\bar u$ and $\bar v$) point $w \in  B(u,\epsilon) \cap B(v,\epsilon) \subseteq G$, and define $\rho(e)$ to be the concatenation of a shortest $1$-chain in  $\IC_\delta^{(1)}$ between $\bar u$ and $\bar w$ and a shortest 1-chain between $\bar w$ and $\bar v$. 
We claim that the length of the $1$-chain $\rho(e)$ is at most $2\epsilon$. Indeed, by construction and Lemma~\ref{claim:equaldist}, we have: 
\begin{align}
\mathrm{length}(\rho(e)) &= d_{\IC_\delta}(\bar u, \bar w) + d_{\IC_\delta}(\bar w, \bar v) = d_G(u, w) + d_G(w, v) \le \epsilon + \epsilon = 2\epsilon. \label{eqn:rholength}
\end{align}


Notice that the restriction $\rho|_{\IC_\delta^{(1)}}:\IC_\delta^{(1)}\rightarrow \IC_\delta^{(1)}$ is the identity mapping. Clearly $\rho|_{\IC_\delta^{(1)}}$ is the identity on the basis elements, the edges in the \Cech{} complex $\IC_\delta^{(1)}$, since there is no shorter path within $\IC_\delta^{(1)}$ than the edge itself. Then, by linearity, $\rho|_{\IC_\delta^{(1)}}$ is the identity on all of $\IC_\delta^{(1)}$.
Additionally, $\rho(\mu_\epsilon^c(\gamma_j))=\rho(\gamma_j)=\gamma_j$
since $\mu_\epsilon^c(\gamma_j)= \gamma_j \in \IC_\epsilon^{(1)}.$ Applying $\rho$ to equation~\eqref{eq:gamma} we obtain the following:
\begin{equation}
\label{eq:forli2}
\ds \rho(\gamma)= \sum_{j=i}^g c_j\gamma_j= \sum_k \rho(t_k).
\end{equation}

Next, we show that for each $k$, $\rho(t_k)$  is the sum of short cycles. Notice that $t_k = [w_0^k, w_1^k, w_2^k]=\partial \Delta_k$ represents a trivial cycle in $\IC^{(1)}_{\epsilon}$, so
there must exist some point $w^k \in \ds \bigcap_{n=0}^2 B(w_{n}^k, \epsilon).$ See Figure~\ref{fig:triangle}.
\begin{figure}[h]
\centering
\includegraphics[width = \textwidth]{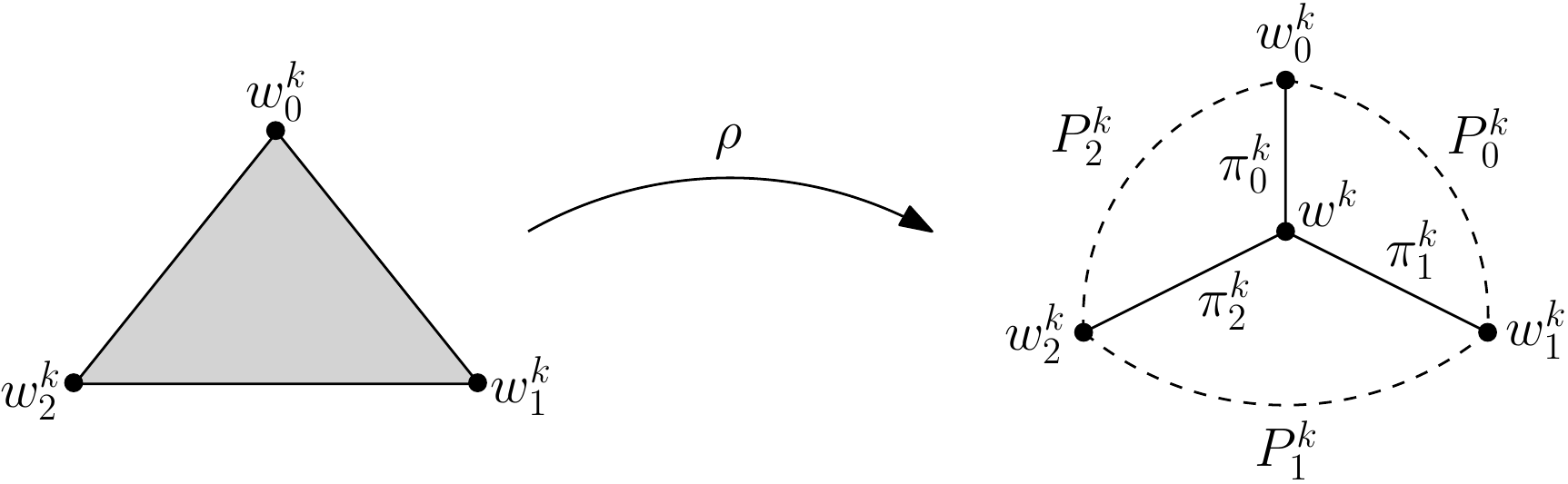}
\caption{Action of $\rho$ on the triangle $t_k = [w_0^k, w_1^k, w_2^k]$. Notice the three cycles contained in $\rho(t_k)$.}
\label{fig:triangle}
\end{figure}

Consequently, for $n=0,1,2$, there exist the following paths $\pi_n^{k}$ and $P_n^k$ in the \Cech{} complex $\IC^{(1)}_{\delta}:$
\begin{itemize}
\item  $\pi_n^k=[w^k,w_n^k]$, which has length less than or equal to $\epsilon$,  and
\item $\ds P_n^k=\rho([w_n^k, w_{(n+1 \mod 3)}^k]),$ which has length at
most $2\epsilon$ by \eqref{eqn:rholength}.
\end{itemize}
Therefore, $\rho(t_k)=\displaystyle \sum_{n=0}^{2} [\pi_n^k+P_n^k-\pi_{(n+1\mod 3)}^k]$ is the sum of three cycles, each of
length at most $\epsilon+\epsilon+2\epsilon = 4\epsilon$. Since the length of  $\rho(t_k)$ is less than $4 \epsilon < \ell_i$, $\rho(t_k)$ can be expressed in terms of the shortest basis
$\ds \{\gamma_j\}_{j=1}^{i-1}$ for $H_1(\IC_\delta): $

\begin{eqnarray}
\rho(\gamma) &=&\ds \sum_k \rho(t_k) = \sum_k \sum_{j=1}^{i-1} c_j^k \gamma_j = \sum_{j=1}^{i-1} c'_j \gamma_j.\\
                       &\implies& \sum_{j=i}^g c_j\gamma_j  \stackrel{\eqref{eq:forli2}}{=}\sum_{j=1}^{i-1}c_j' \gamma_j \\
                       &\implies&   \sum_{j=1}^{i-1} c_j \gamma_j +\sum_{j=i}^g (-c'_j)\gamma_j = 0.
\end{eqnarray}
As the set $\ds \{\gamma_i\}_{i=1}^g$ is a shortest basis for $H_1(\IC_\delta)$, the coefficients in the above sums must all be zero, that is $c_j = 0$ for all $j$, which contradicts our initial assumption.
Therefore, the set $\{[\mu_\epsilon^c(\gamma_i)], [\mu_\epsilon^c(\gamma_{i+1})],\ldots,
[\mu_\epsilon^c(\gamma_g)]\}$ is linearly independent in $H_1(\IC_\epsilon)$. In particular, $\gamma_i$ does not become trivial before $\displaystyle\frac{\ell_i}{4}$.

Next, to prove B), we show that the map $\mu_\epsilon^h : H_1(\IC_\delta) \rightarrow H_1(\IC_\epsilon)$ is surjective by showing that it has a right inverse up to homotopy.
In particular, we will show that for every  $[\eta] \in H_1(\IC_\epsilon)$,
\begin{equation}\label{LL}
\mu^{h}_{\epsilon}([\rho(\eta)])=[(\mu_\epsilon^c \circ \rho)(\eta)] = [\eta] \in H_1(\IC_\epsilon) \end{equation}
where the chain $\eta\in \IC_\epsilon^{(1)}$ is a geometric realization of the class $[\eta]$.

Consider a cycle $\eta= \langle u_0, u_1, \ldots, u_k, u_0\rangle \in \IC_\epsilon^{(1)}$ representing $[\eta] \in H_1(\IC_\epsilon).$
Set $p_j = \rho([u_j, u_{j+1}]) = \langle u_j, v_1^j,\ldots, v_{m_j}^j, u_{j+1}\rangle$, for $j=0,1,\ldots,k$ and $u_{k+1}=u_0.$  
Then the image $\rho(\eta)$ is just a concatenation of these paths  $\rho(\eta)=  p_0 +  p_1 +  \cdots + p_k  \in \IC_\delta^{(1)}$.
Since $\mu_\epsilon^c$ is induced by inclusion, we abuse the  notation slightly and use $p_0+p_1 +\cdots + p_k$ to denote the image $\mu_\epsilon^c(\rho(\eta))$ in $\IC_\epsilon^{(1)}$ as well. 

To show equation~\eqref{LL} holds, we need to prove that $[\mu_\epsilon^c(\rho(\eta))]_{\IC_\epsilon}=[\eta]_{\IC_\epsilon}$, which is achieved by showing that $p_j$ is path homotopic to  edge $[u_j, u_{j+1}]$ in $\IC_\epsilon$ for all $j=0,1,\ldots,k$ and $u_{k+1}=u_0.$ 


To this end, note that by construction of $\rho([u_j, u_{j+1}])$, the path $p_j$ is (i) either the edge $[u_j, u_{j+1}]$; or (ii) consists of two shortest 1-chains $p_j^{(1)}$ and 
$p_j^{(2)}$, where $p_j^{(1)}$ and 
$p_j^{(2)}$ are from $u_j$ to $v_l^j$ and $v_l^j$ to
$u_{j+1}$, respectively (see Figure~\ref{fig:fig2}) -- here, $v_l^j$ corresponds to the node $\bar w$ in the definition of $\rho(e)$ earlier. 
If it is case (i), then obviously, $p_j$ is homotopic to edge $[u_j, u_{j+1}]$ (as they are the same). 
Now consider case (ii). Note by definition of $\rho$, $d_{\IC_\delta}(u_j, v_l^j) \le \epsilon$ and $d_{\IC_\delta}(v_l^j, u_{j+1}) \le \epsilon$. As $p_j^{(1)}$ (resp. $p_j^{(2)}$) is a shortest path between $u_j$ and $v_l^j$ (resp. between $v_l^j$ and $u_{j+1}$), it then follows that $d_G(u_j, v_n^j) = d_{\IC_\delta}(u_j, v_n^j) \le \epsilon$ for all $0 \leq n < l$. Similarly, $d_G(u_{j+1}, v_n^j) \le \epsilon$ for  all $l\leq n < m_n$. 
Therefore, each of the following is a  $2$-dimensional simplex:  the triangle $[u_j, v_n^j, v_{n+1}^j] \in
\IC_\epsilon$ for all $0 \leq n < l$ and the triangle $[u_{j+1}, v_n^j, v_{n+1}^j]\in
\IC_\epsilon$ for all $l\leq n < m_n$.
It then follows that the path (1-chain) $p_i$ is homotopic to the edge $[u_j, u_{j+1}]$ for case (ii), as well. Concatenating these homotopies proves the homotopy equivalence of $\eta$ and $\rho(\eta)$. Hence, $[\mu_\epsilon^c(\rho(\eta))] = [\eta]$ which establishes equation~\eqref{LL}. 

\begin{figure}[ht]
\centering
\includegraphics[scale=0.8]{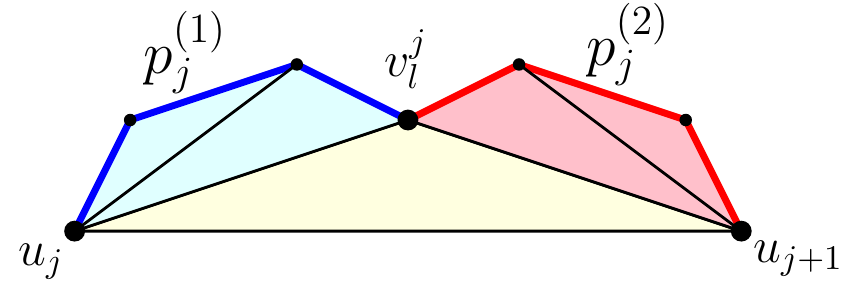}
\caption{ A part of $\IC_{\delta}$ used to illustrate that $[\rho(\eta)]=[\eta].$ In particular, each edge $[u_j, u_{j+1}]$ will be mapped by $\rho$ to a chain of edges. The path $p_j^{(1)}$ is colored in blue, and $p_j^{(2)}$ is colored in red. The homotopy is realized in two $2$-dimensional simplices (represented by the blue/red shading) that exist in $\IC_{\delta}$ based  on the \Cech{} construction.}
\label{fig:fig2}
\end{figure}

 Notice that  $[\rho(\eta)] = \ds \sum_{j=i}^g c_j [\gamma_j] \in H_1(\IC_\delta)$
since $\ds\frac{\ell_{i-1}}{4} \leq \epsilon < \frac{\ell_i}{4}$.  By equation~\eqref{LL} we have
\begin{eqnarray*}
[\eta]=\mu^{h}_{\epsilon}([\rho(\eta)])& =& \ds \mu_\epsilon^h(\sum_{j=i}^g c_j [\gamma_j])\\
&=& \sum_{j=i}^g c_j [\mu_\epsilon^c(\gamma_j)]\in \text{Span}(\{[\mu_\epsilon^c(\gamma_j)]\}_{j\geq i}),
\end{eqnarray*}
which completes the proof of the surjectivity of $\mu^{h}_{\epsilon}.$ This establishes the spanning condition B). 
In other words, if $[\eta]$ is a homology class in $H_1(\IC_\epsilon)$ then it must be formed only from homology classes $[\mu_\epsilon^c(\gamma_j)]$ for $j \geq i$, and thus no additional cycles are created.
\qed
\end{proof}

\section{Future Work}
\label{sec:future}

The overarching theme of this work is to show how persistence may be used to obtain qualitative-quantitative summaries of
metric graphs that reflect the underlying topology
of the graphs. We obtained a complete characterization of all possible intrinsic \Cech{}
persistence diagrams in homological dimension one for metric graphs. What is currently known regarding the characterization of intrinsic \Cech{} persistence diagrams for metric graphs is summarized in a diagram shown in Figure~\ref{fig:IC-Future}. The horizontal axis represents the homological dimension and the vertical axis represents the genus of a graph. In this setting,
the previous results of Adamaszek, et al. \cite{AdamaszekAdams2015} who consider the intrinsic \Cech{} persistence diagrams in all
dimensions for a graph that consists of a single loop, lie on the horizontal strip at height one, while the results in this paper constitute the blue vertical strip.
The rest of the upper-right quadrant is unknown, and our hope is to make 
further progress toward a complete characterization of the intrinsic \Cech{} persistence diagrams associated to arbitrary metric graphs. Moreover, we aim to generalize our results
to the Vietoris-Rips complex.
\begin{figure}[ht]
\centering
\includegraphics[width=0.75\linewidth]{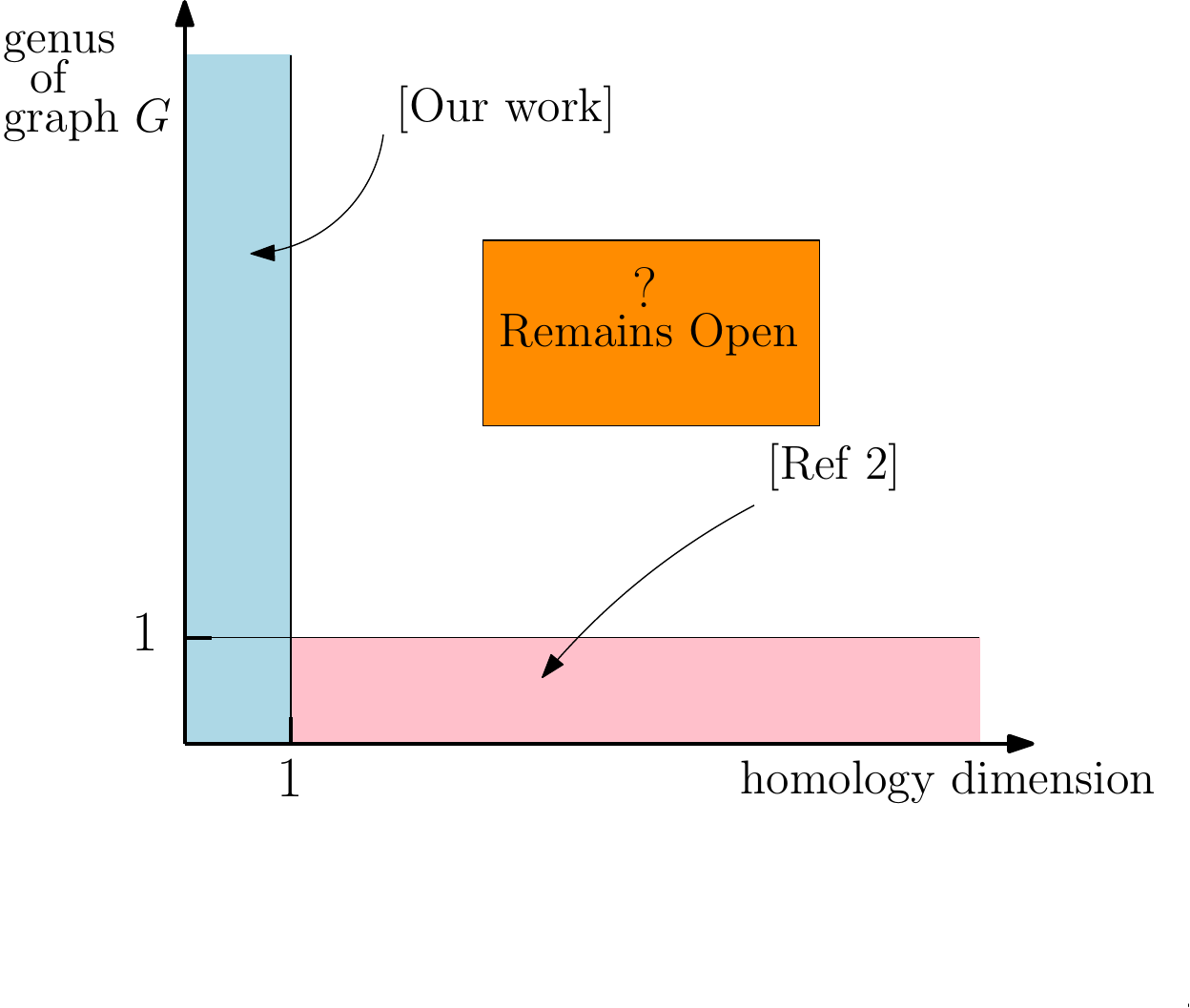}
\vspace{-1.25 cm}
\caption{Figure summarizing the results from this paper and from \cite{AdamaszekAdams2015}.}
\label{fig:IC-Future}
\end{figure}

The choice of a particular complex may be inspired by particular graph features that one is interested in. A \emph{graph motif} is usually thought of as a graph on a small number of vertices (in general, any graph pattern can be a motif). Counting the number of small motifs in a graph is equivalent to the subgraph isomorphism problem, which is NP-complete. Since persistence has a polynomial time computational complexity, the question we would like to answer is: can the intrinsic \Cech{} or other related persistence diagrams be used to determine or approximate graph motif counts? Additionally,  the local version of this question,  the number of graph motifs incident with a particular vertex, may be approached via the local homology at a vertex (homology of the $k$-neighborhood of a vertex relative to its boundary). As a start, persistence-based characterizations of a class of graph motifs should be obtained. Depending on the type of characterization obtained, we would be interested in determining to what extent our persistence-based summaries could be useful in the classification of the motifs present in a query graph.

Ultimately, the complete or partial characterization of the topological information about a graph that is captured by persistent homology associated to various chain complex constructions is closely related to comparing their discriminative powers. In particular, we are interested in comparing the \Cech~and persistence distortion distance summaries.

The intrinsic \Cech{} filtration and associated persistence diagrams allow one to define the \emph{intrinsic
\Cech{} distance}, $d_{IC}$, between two metric graphs $(G_1, d_{G_1})$ and $(G_2,
d_{G_2})$. This distance, introduced in \cite{ChazalSilvaOudot2014}, is defined as follows:
\[ d_{IC}(G_1, G_2) := d_B(Dg_1 IC_{G_1}, Dg_1 IC_{G_2}), \]
where $d_B$ is the bottleneck distance between the two intrinsic \Cech{}
persistence diagrams in dimension 1.

The \emph{persistence distortion distance}, $d_{PD}$, that was first introduced in \cite{DeyShiWang2015}, is more closely related to the metric properties of a graph. Given a base point $s \in |G|$, define $f_s : |G| \rightarrow \R$ to be the geodesic distance to the base point $s$, i.e,
$f_s(x) = d_G(s,x)$ for all $x \in |G|$. Then $Dg_1 f_s$ is the 1st-extended
persistence diagram \cite{Cohen-Steiner2009} associated to the sublevel set filtration induced by $f_s$. One may do this for any given
base point in the metric graph, yielding a set of persistence
diagrams for each graph. Let
\begin{align*}
\phi : |G| &\rightarrow Sp Dg\\
       s   &\mapsto Dg f_s
\end{align*}
where $Sp Dg$ denotes the space of all persistence diagrams. Then $\phi(|G|) \subset
Sp Dg$ is called the \emph{persistence distortion} of $G$. The
\emph{persistence distortion distance} between two metric graphs is defined to be the
Hausdorff distance between their persistence distortion sets:
\[ d_{PD}(G_1, G_2) := d_H(\phi(|G_1|), \phi(|G_2|)). \]

A natural question to ask is whether or not $d_{PD}$ is more discriminative than $d_{IC}$, i.e., whether or not there exists a constant $c > 0$ such that
\[ d_{IC}(G_1, G_2) \leq c \cdot d_{PD}(G_1, G_2).\]
We are currently working on extending preliminary results that establish the inequality for certain classes of metric graphs to arbitrary input graphs.

\section*{Acknowledgements}
We are grateful for the Women in Computational Topology (WinCompTop) workshop for initiating our research collaboration. In particular, participant travel support was made possible through an NSF grant (NSF-DMS-1619908), and some additional travel support and social outings throughout the workshop were made possible through a gift from Microsoft Research. The Institute for Mathematics and its Applications generously offered the use of their facilities, their experienced staff to facilitate conference logistics, and refreshments. We appreciate their continued support of the applied algebraic topology community with regard to the WinCompTop Workshop, the special thematic program \emph{Scientific and Engineering Applications of Algebraic Topology} held during the 2013-2014 academic year, and the Applied Algebraic Topology Network WebEx talks. Finally, we would like to thank the AWM ADVANCE grant for travel support for organizers and speakers to attend the WinCompTop special session at the AWM Research Symposium in April 2017.

EP was partially supported by the Asymmetric Resilient Cybersecurity Initiative at Pacific Northwest National Laboratory,
part of the Laboratory Directed Research and Development Program at PNNL, a multi-program national laboratory operated by
Battelle for the U.S. Department of Energy.
During the completion of this project, RS was partially supported by the Simons Collaboration Grant,
BW was partially supported by NSF-IIS-1513616, and
YW was partially supported by NSF-CCF-1526513 and NSF-CCF-1618247.

%

\bibliographystyle{plain}
\bibliography{references}

\begin{thebibliography}{10}

\bibitem{AanjaneyaChazalChen2012}
Mridul Aanjaneya, Fr\'{e}d\'{e}ric Chazal, Daniel Chen, Marc Glisse, Leonidas
  Guibas, and Dmitriy Morozov.
\newblock Metric graph reconstruction from noisy data.
\newblock {\em International Journal of Computational Geometry \&
  Applications}, 22(04):305--325, 2012.

\bibitem{AdamaszekAdams2015}
Michal Adamaszek and Henry Adams.
\newblock The {V}ietoris-{R}ips complexes of a circle.
\newblock ar{X}iv 1503.03669, 2015.

\bibitem{AdamaszekAdamsFrick2016}
Micha{\l} Adamaszek, Henry Adams, Florian Frick, Chris Peterson, and Corrine
  Previte-Johnson.
\newblock Nerve complexes of circular arcs.
\newblock {\em Discrete {\&} Computational Geometry}, 56(2):251--273, 2016.

\bibitem{BiswalYetkinHaughton1995}
Bharat Biswal, F.~Zerrin Yetkin, Victor~M. Haughton, and James~S. Hyde.
\newblock Functional connectivity in the motor cortex of resting human brain
  using echo-planar {MRI}.
\newblock {\em Magnetic Resonance in Medicine}, 34:537--541, 1995.

\bibitem{carlsson2009topology}
Gunnar Carlsson.
\newblock Topology and data.
\newblock {\em Bulletin of the American Mathematical Society}, 46(2):255--308,
  2009.

\bibitem{ChazalSilvaOudot2014}
Fr\'{e}d\'{e}ric Chazal, Vin de~Silva, and Steve Oudot.
\newblock Persistence stability for geometric complexes.
\newblock {\em Geometriae Dedicata}, 173:193--214, 2014.

\bibitem{Cohen-Steiner2009}
David Cohen-Steiner, Herbert Edelsbrunner, and John Harer.
\newblock Extending persistence using poincar{\'e} and lefschetz duality.
\newblock {\em Foundations of Computational Mathematics}, 9(1):79--103, 2009.

\bibitem{deSilva_Ghrist_AGT:2007}
Vin de~Silva and Robert Ghrist.
\newblock {Coverage in sensor networks via persistent homology}.
\newblock {\em Algebraic \& Geometric Topology}, 7:339--358, 2007.

\bibitem{DeyShiWang2015}
Tamal~K. Dey, Dayu Shi, and Yusu Wang.
\newblock Comparing graphs via persistence distortion.
\newblock In Lars Arge and J{\'a}nos Pach, editors, {\em Proceedings 31st
  International Symposium on Computational Geometry}, volume~34 of {\em Leibniz
  International Proceedings in Informatics (LIPIcs)}, pages 491--506, Dagstuhl,
  Germany, 2015. Schloss Dagstuhl--Leibniz-Zentrum fuer Informatik.

\bibitem{EdelsbrunnerHarer2008}
Herbert Edelsbrunner and John Harer.
\newblock Persistent homology - a survey.
\newblock {\em Contemporary Mathematics}, 453:257--282, 2008.

\bibitem{citeulike:2795523}
Robert Ghrist.
\newblock {Barcodes: The persistent topology of data}.
\newblock {\em Bulletin of the American Mathematical Society}, 45:61--75, 2008.

\bibitem{hatcher2002algebraic}
Allen Hatcher.
\newblock {\em Algebraic Topology}.
\newblock Cambridge University Press, 2002.

\bibitem{Kuchment2004}
Peter Kuchment.
\newblock Quantum graphs {I}. {S}ome basic structures.
\newblock {\em Waves in Random Media}, 14(1):S107--S128, 2004.

\bibitem{massey1991basic}
W.S. Massey.
\newblock {\em A Basic Course in Algebraic Topology}.
\newblock Graduate Texts in Mathematics. Springer New York, 1991.

\bibitem{munkres1984elements}
J.R. Munkres.
\newblock {\em Elements of Algebraic Topology}.
\newblock Advanced book classics. Perseus Books, 1984.

\bibitem{Previte2014}
Corrine Previte.
\newblock {\em The $\mathcal{D}$-Neighborhood Complex of a Graph}.
\newblock PhD thesis, Colorado State University, 2014.

\bibitem{rotman2009introduction}
J.~Rotman.
\newblock {\em An Introduction to Homological Algebra}.
\newblock Universitext. Springer New York, 2009.

\bibitem{Taylan2016}
Demet Taylan.
\newblock Matching trees for simplicial complexes and homotopy type of devoid
  complexes of graphs.
\newblock {\em Order}, 33:459--476, 2016.

\end{thebibliography}

\end{document}